\documentclass{amsart}
\usepackage{tikz-cd}
\usepackage{hyperref}
\usepackage{amsmath}
\usepackage{amssymb}
\usepackage{mathtools}
\usepackage{amsthm}
\usepackage{bm}
\usepackage{xcolor}
\usepackage{accents}
\usepackage[shortlabels]{enumitem}
\usepackage{stackrel}
\usepackage{stmaryrd}
\usepackage{mathrsfs}
\usepackage{cleveref}
\title{Adjoint $L$-functions, congruence ideals, and Selmer groups over $\GL_n$}
\author{Ho Leung Fong} 
\email{hlfong1@sheffield.ac.uk}
\address{School of Mathematical and Physical Sciences, University of Sheffield}
\newtheorem{theorem}{Theorem}[section]

\theoremstyle{definition}
\newtheorem{definition}[theorem]{Definition}

\newtheorem{lemma}[theorem]{Lemma}

\newtheorem{proposition}[theorem]{Proposition}

\newtheorem{corollary}[theorem]{Corollary}

\newtheorem{defn/lemma}[theorem]{Definition/Lemma}

\newtheorem*{claim}{Claim}

\theoremstyle{remark}
\newtheorem{remark}[theorem]{Remark}
\newcommand{\GL}{\mathrm{GL}}

\newcommand{\A}{\mathbb{A}}
\newcommand{\Q}{\mathbb{Q}}
\newcommand{\R}{\mathbb{R}}
\newcommand{\C}{\mathbb{C}}
\newcommand{\Z}{\mathbb Z}

\newcommand{\m}{\mathfrak{m}}
\newcommand{\n}{\mathfrak{n}}
\newcommand{\p}{\mathfrak p}

\newcommand{\T}{\mathbb T}

\renewcommand{\O}{\mathcal O}
\renewcommand{\H}{\mathcal H}

\def\lf{\left\lfloor}   
\def\rf{\right\rfloor}

\DeclareMathOperator{\Gal}{Gal}

\DeclareMathOperator{\ad}{Ad}
\DeclareMathOperator{\Ad}{Ad}

\DeclareMathOperator{\Res}{Res}
\DeclareMathOperator{\Frob}{Frob}

\DeclareMathOperator{\Hom}{Hom}
\DeclareMathOperator{\Spec}{Spec}

\DeclareMathOperator{\End}{End}

\begin{document}
\begin{abstract}
    In this paper, we relate $L(1,\pi,\Ad^\circ)$ to the congruence ideals for cohomological cuspidal automorphic representations $\pi$ of $\GL_n$ over any number field. We then use this result to deduce relationships between the congruences of automorphic forms and adjoint $L$-functions. For CM and totally real fields, we apply the result to obtain a lower bound on the cardinality of certain Selmer groups in terms of $L(1,\pi,\Ad^\circ)$.
\end{abstract}
\maketitle
\tableofcontents
\section{Introduction}
In number theory, we study many different types of $L$-functions, such as $L$-functions of elliptic curves, Artin $L$-functions, and Dirichlet L-functions. A recurring theme is that the special values of these $L$-functions often encode deep arithmetic information. Some notable examples include the analytic class number formula, the Herbrand-Ribet theorem, and the Birch-Swinnerton-Dyer conjecture.

The focus of this paper is the adjoint $L$-function of an automorphic representation. Special values of adjoint $L$-functions frequently reflect congruences between automorphic forms, and these congruences can in turn be measured by the so-called congruence ideal. The relationship between adjoint $L$-values and congruence ideals has by now become a central theme in modern number theory; for a historical overview, we refer the reader to the introduction of \cite{Raghuram-Balasubramanyam}. Here we offer a complementary motivation.

In 1994, Andrew Wiles and Richard Taylor proved Fermat's last theorem by showing that every semistable elliptic curve is modular. The key is to establish a modularity lifting theorem, which in their setup boils down to showing that the universal deformation ring $R$ of some residual Galois representation is isomorphic to a suitable Hecke algebra $T$.
To do so, they used the Taylor-Wiles method to prove a modularity lifting theorem in the minimally ramified case \cite{Taylor-Wiles}. To extend the theorem to the non-minimal case,  Wiles \cite{wiles_modular_1995} used a numerical criterion for ring isomorphism. To apply the criterion involves studying the tangent space of $R$ on the one hand and examining an invariant called the congruence ideal $\eta_T$ of $T$ on the other hand.

It has since proved fruitful to study congruence ideals not only for Hecke algebras but also for various cohomology modules. In \cite{tu}, the congruence ideals associated with the cohomology of certain locally symmetric spaces are related to the value at $s=1$ of the adjoint $L$-function for $\GL_2$ over specific number fields. Using the work of \cite{Raghuram-Balasubramanyam}, we generalize these results to $\GL_n$ over an arbitrary number field.

Let $\pi$ be a regular algebraic cuspidal automorphic representation of $\GL_n(\mathbb A_F)$, where $F$ is a number field.
Our first main result is \Cref{thm: main}, which establishes that, under suitable assumptions, the (cohomological) congruence ideal of $\pi$ is generated by a special value of the normalised adjoint $L$-function $L^{alg}(1,\pi,\Ad^0,\epsilon)$.


Building on Theorem \ref{thm: main}, we derive consequences for congruences between automorphic forms. In \Cref{coro: congruence,coro: cuspidal congruence}, we show that under suitable assumptions, if the $p$-adic valuation of $L(1,\pi,\Ad^\circ,\epsilon)$ is positive, then there is an automorphic representation $\pi'\not\cong \pi$ whose Hecke eigensystem is congruent to that of $\pi$.

In the last subsection, we provide in \Cref{thm: selmer} a lower bound for the cardinality of a certain Selmer group in terms of special values of $L$-functions when $F$ is a $CM$ field. This can be viewed as partial progress on the Bloch-Kato conjecture.

\subsection{Strategy}
We briefly summarise the strategy of the paper, omitting certain technical details in order to emphasise the main ideas.

The starting point is to realise automorphic representations in cohomology and to work at the level of cohomology. By results of Borel and Franke \cite{Borel-regularization,Franke}, any regular algebraic cuspidal automorphic representation of $\GL_n(\mathbb A_F)$ contributes to the cohomology of the locally symmetric space $X_U$ associated with $\GL_n$. Poincar\'e duality furnishes a perfect pairing between the singular cohomology and the compactly supported cohomology of $X_U$ with $\Z_p$-coefficients. Under the additional assumption that the cohomology of the boundary of $X_U$ is $p$-torsion free, this pairing induces a perfect pairing on the inner cohomology of $X_U$.

By \Cref{lem: pp}, any such perfect pairing that is suitably compatible with the action of the Hecke algebra yields an explicit formula for the corresponding congruence ideal: the ideal is generated by the value of the pairing on appropriate $\Z_p$-generators.
In our setting, the relevant pairing is the cup product on cohomology, which, on the automorphic side, corresponds essentially to the Petersson inner product of two automorphic forms. Via the Rankin-Selberg method, this inner product can be expressed in terms of the adjoint $L$-function. This leads directly to a formula for the congruence ideal in terms of the value $L(1,\pi,\Ad^0,\epsilon)$ in \Cref{thm: main}.

Using the relationship between congruence ideals and congruences of automorphic representations (\Cref{lem:congruence}), we then obtain criteria for the existence of congruent automorphic representations in terms of adjoint $L$-values (\Cref{coro: congruence}). Finally, invoking the local-global compatibility results for Galois representations established in \cite{Caraiani-Newton}, together with standard techniques from Galois deformation theory, we derive in \cref{sec: selmer} (when $F$ is a CM or totally real field) a lower bound for the order of a certain Selmer group in terms of special values of adjoint $L$-functions.

\subsection{Notation}
For a number field $F$, we write $\mathbb A_F$ for the ring of adeles, $F_\infty:=F\otimes_\Q \R=\prod_{v\mid \infty} F_v,$ $$\mathbb A_F^\infty$$ the ring of finite adeles, and $\mathbb A_F^{\infty,S}$ for the ring of finite adeles without the components at $S$.
The contragredient of an automorphic representation of $\GL_n(\mathbb A_F)$ is denoted by $\tilde\pi$.

If $G$ is a locally profinite group and $U$ is an open compact subgroup of $G$, then we let 
\begin{align*}
    \mathcal H(G,U) &:=\{U\text{-biinvariant compactly supported functions }G\to \Z\}\\
    &\cong \Z[U\backslash G/U],
\end{align*}
where multiplication is given by convolution with respect to the Haar measure on $G$ which gives $U$ volume $1$.
If $F$ is a number field and $\O$ is the ring of integers of a finite extension of $\Q_p$, then $$\T^S:= \mathcal H(\GL_n(\A_F^{\infty,S}),\prod_{v\notin S\cup\{\infty\}}\GL_n(\O_{F_v}))\otimes_\Z \O.$$ (This of course depends on $F, \O$.)

We write $\Ad^0$ for the (trace zero) adjoint representation, which is the representation of $\GL_n(\C)$ on the $n\times n$ trace zero matrices by conjugation.

All $L$-functions include the $L$-factors at infinity unless otherwise stated.

If $F$ is a number field and $G=\GL_{n/F}$, then $K_\infty$ will mean the product of $A_G:=\R_{>0}$ \footnote{$A_G=\R_{>0}$ is embedded diagonally into the centre of $G(F_\infty).$}and the standard maximal compact subgroup of $G(F\otimes_\Q\R)$, so $$K_\infty\cong \R_{>0} \cdot (O_n(\R)^{r_1}\times U_n(\R)^{r_2}),$$ where $r_1,r_2$ are the number of real and complex places of $F$ respectively and $U_n(\R):=\{g\in \GL_n(\C): \bar g^T g=1\}$ is the unitary group.
Also, $\mathfrak g$ will be the Lie algebra of $\GL_n(F_\infty)$.

\section{Cohomology}
Fix a number field $F$ with $r_1$ real places and $r_2$ complex places for the entire section.
Readers not very familiar with the relationship between automorphic representations and cohomology are recommended to read the excellent papers \cite{Schwermer,Clozel}. 
\subsection{Cohomology and Hecke operators}\label{lss}
We shall mostly follow \cite[section 6]{KT} and \cite[section 2.1]{Hansen-eigenvarieties} to define the cohomology group and Hecke operators.
Most of the material is standard, but there are many different variations, so we think it is necessary to state clearly our conventions.

Let $G=\GL_n$.
For an open compact subgroup $U\subset G(\mathbb A_F^\infty)$, we define\footnote{We quotient out by $K_\infty^{\circ}$ rather than $K_\infty$ because we would like to realise cohomological automorphic representations in cohomology. If we instead quotient out by $K_\infty$, then some of these representations only appear in the relative Lie algebra after twisting by a character $\epsilon$ of $\widehat{K_\infty/K_\infty^\circ}$, but the fact that $\epsilon$ is not a character of $U\subset G(\A_F^\infty)$ makes it unclear how to define the inner cohomology with coefficient in $M_\lambda\otimes\epsilon$. See however \cite[Proposition 3.2.5, Theorem 4.3.5]{januszewski} for a possible way to do this using locally algebraic representations.}
\[X_U:=G(F)\backslash G(\mathbb A_F)/K_{\infty}^\circ U.\]
and 
\[X:= G(F_\infty)/K_{\infty}^\circ.\]
Note that $X_U=G(F)\backslash (X\times G(\mathbb A_F^\infty)/U).$

We call an element $g=(g_v)_v\in G(\A^\infty_F)$ \emph{neat} if $\cap_v \Gamma_v=\{1\}$, where $\Gamma_v\subset \bar\Q^\times$ is the torsion subgroup of the subgroup generated by the eigenvalues of $g_v$.
We call an open compact subgroup $K\subset G(\A^\infty_F)$ \emph{neat} if all of its elements are neat.
\begin{definition}\label{defn: good}
    We call $U\subset G(\mathbb A_F^\infty)$ a \emph{good subgroup} if it is a neat subgroup of the form $U=\prod_v U_v\subset\prod \GL_n(\mathcal O_{F_v})$.
\end{definition}

Let $U$ be a good subgroup and $M$ be a $\Z[U]$-module.
We define a locally constant sheaf $\mathcal L_M$ on $X_U$ as the sheaf of continuous sections of the map
\[G(F)\backslash (X\times G(\mathbb A_F^\infty)\times M)/U \to X_U\]
where $G(F)\times U$ acts on $X\times G(\mathbb A_F^\infty)\times M$ by $(\gamma,u)(x,g,m)=(\gamma x,\gamma g u^{-1},um)$ and $M$ is equipped with the discrete topology.
We define $$H^*(X_U,M):=H^*(X_U,\mathcal L_M)$$ to be the sheaf cohomology.

\begin{proposition}\cite[Proposition 6.2]{KT}
    If $U$ is a neat\footnote{In \cite{KT}, this is stated for good subgroups, but the proof actually works for all neat subgroups.} subgroup, then 
    \[H^*(X_U,M)\cong H^*(C_{\mathbb A}^{\bullet}(U,M)),\]
    where $$C_{\mathbb A}^\bullet(U,M):=\Hom_{G(F)\times U}(C_{{\mathbb A},\bullet}, M)$$ and $C_{{\mathbb A},\bullet}$ is the group of singular chains on $X\times G(\mathbb A_F^\infty)$ with $\Z$ coefficients.
\end{proposition}

To define the action of the Hecke algebra, we suppose $M$ is actually a left $\Z[\Delta]$-module, where $\Delta\subset G(\A_F^\infty)$ is a submonoid containing $U$.
Note that the compactness of $U$ implies that $U\subset\Delta$ is a Hecke pair.
Let $\mathcal H(\Delta,U)$ be the set of locally constant, compactly supported functions $f:\Delta\to \Z$ which are $U$-biinvariant.
We can (and will) regard it as a subalgebra of $\mathcal H(G(\A^\infty_F),U)$.

For $\delta\in \Delta$,  let the characteristic function $[U\delta U]$ acts on the complex $C_{\mathbb A}^\bullet(U,M)$ by 
\[([U\delta U]^* \phi)(\sigma)=\sum \delta_i\phi(\delta_i^{-1}\sigma)\]
where $U \delta U=\bigsqcup_i \delta_i U$, $\phi\in C_{\mathbb A}^\bullet(U,M)$ and $\sigma\in C_{{\mathbb A},\bullet}.$
This is independent of the choices of $\delta_i$.
By taking cohomology, we get an action of $\mathcal H(\Delta,U)$ on $H^*(X_U,M)$.

Let $T_n$ and $B_n$ be the standard diagonal torus and Borel subgroup of $\GL_{n/\Z}$.
Let $w_0$ be the longest element of the Weyl group\footnote{If we write the characters of $T_n$ as $\lambda=(\lambda_1,\dots,\lambda_n)$, then $w_0\lambda=(\lambda_n,\dots,\lambda_1)$}.
\begin{definition}\cite[Definition 2.1]{Geraghty}
    Let $A$ be a commutative ring. If $\lambda\in \Z^n$ is a dominant weight for $\GL_n$, then we define the algebraic representation $Ind^{\GL_n}_{B_n}(w_0\lambda)_{/A}$ of $\GL_{n/A}$ to be
    \begin{equation*}
        \{f\in A[\GL_n]: f(bg)=(w_0\lambda)(b) f(g) \text{ for all $A$-algebras }B, g\in\GL_n(B), b\in B_n(B)\}
    \end{equation*}
    where $A[\GL_n]:=Mor_{\Spec A}(\GL_{n/A},\mathbb A^1_{A})$ \footnote{Here, $\mathbb A^1_A=\Spec A[T]$ is the affine line over $A$.}and $\GL_{n/A}$ acts by right translation.
    We let $$M_{\lambda,A}:= Ind^{\GL_n}_{B_n}(w_0\lambda)_{/A}(A),$$
    which is a representation of $\GL_n(A)$.
\end{definition}
If $E$ is a $p$-adic field with ring of integers $\mathcal O$, then from \cite[page 1349]{Geraghty}, $M_{\lambda,\mathcal O}$ is finite free over $\mathcal O$.
Also, $M_{\lambda,\mathcal O}\otimes_{\mathcal O} E=M_{\lambda,E}$ is the algebraic representation of $\GL_n(E)$ of highest weight $\lambda$. By \cite[page 19]{Newton-Thorne-derived}, for all $\O$-algebras $R$, the natural map $M_{\lambda,\mathcal O}\otimes_{\O} R \to M_{\lambda,R}$ is an isomorphism.

We write $\Z^n_+:=\{(\lambda_1,\dots,\lambda_n)\in\Z^n:\lambda_1\ge \dots\ge \lambda_n\}$.
Let $E$ be a finite extension of $\Q_p$ inside $\overline \Q_p$ which contains all embeddings of $F$ to $\overline\Q_p,$ $\O$ be the ring of integers of $E$, and $\mu\in (\Z^n_+)^{\Hom(F,E)}$.

We define the $\O$-module $$M_{\mu}:=M_{\mu,\O}:=\bigotimes_{\tau\in \Hom(F,E)} M_{\mu_\tau ,\O}$$ which receives an action of $\prod_{v\mid p}\GL_n(\O_{F_v})$ by
 $(g_v)_v \cdot \otimes m_\tau = \otimes g_{v(\tau)}m_{\tau},$ where $v(\tau)$ is the place of $F$ induced by $\tau.$
Then $\GL_n(\A_F^{\infty,p})\times U_p$ acts on $M_{\mu,\O}$ by projection to $U_p:=\prod_{v\mid p}U_v$.
By the formalism above\footnote{Strictly speaking, the formalism gives us an action of $\mathcal H(\GL_n(\A_F^{\infty,p})\times U_p,U).$ Yet, this algebra is canonically isomorphic to $\mathcal H(\GL_n(\A_F^{\infty,p}),U^p)$}, $\mathcal H(\GL_n(\A_F^{\infty,p}),U^p)$ acts on $H^*(X_U,M_{\mu,\O})$.

We define $M_{\mu,E}:=\otimes_{\tau\in \Hom(F,E)} M_{\mu_\tau ,E}$ as an $E[\prod_{v\mid p}\GL_n(F_v)]$-module.
Then $\GL_n(\A_F^{\infty})$ acts on $M_{\mu,E}$ by projection to $\GL_n(F_p):=\GL_n(\prod_{v\mid p}F_v)$.
By the formalism above, $\mathcal H(\GL_n(\A_F^{\infty}),U)$ acts on $H^*(X_U,M_{\mu,E})$.
This is compatible with the construction above, i.e. we have an isomorphism $H^*(X_U,M_{\mu,\O})\otimes_\O E \cong H^*(X_U,M_{\mu,E})$ that is Hecke-equivariant for the restriction map $\mathcal H(\GL_n(\A_F^{\infty,p}),U^p) \to \mathcal H(\GL_n(\A_F^{\infty}),U)$.

Fix an isomorphism $\iota:\overline\Q_p\xrightarrow{\sim}\C$ for the rest of this section. We define $M_{\mu,\C}:=\otimes_{\tau\in \Hom(F,E)} M_{\mu_\tau ,\C}$.
This is acted on by $\prod_{v\mid p}\GL_n(F_v)$ via $\iota$ and also by 
$G(F_\infty)$, where $F_\infty:=F\otimes_\Q \R,$  by 
$$G(F_\infty)\hookrightarrow G(F\otimes_\Q \C)=\prod_{\tau\in \Hom(F,\C)} G(\C)\curvearrowright \otimes_{\tau\in \Hom(F,E)} M_{\mu_\tau ,\C}$$ using the identification $\Hom(F,E)= \Hom(F,\C)$ given by $\iota$.
Note that the 2 actions of $G(F)$ on $M_{\mu,\C}$ agree.
As $M_{\mu,\C}$ is an irreducible representation of $\GL_n(F\otimes_\Q \C)$, it has a central character.
In particular, $A_G$ acts\footnote{Recall that $A_G=\R_{>0}$ is embedded diagonally into the centre of $G(F_\infty).$} by a character
\[\chi^{-1}:A_G\to \C^\times\]
on $M_{\mu,\C}.$
One can show that for all good subgroup $U\le G(\A_F^\infty)$,
\[H^*(X_U,M_{\mu,\C})\cong H^*(\mathfrak g,K_\infty^\circ ,C^\infty(G(F)\backslash G(\A_F)/U,\chi)\otimes_\C M_{\mu,\C}),\]
where $\mathfrak g=Lie(G(F_\infty))$ and 
$$C^\infty(G(F)\backslash G(\A_F)/U,\chi)$$ is the set of smooth functions $f:G(F)\backslash G(\A_F)/U \to \C$ with $f(ag)=\chi(a)f(g)$ for all $a\in A_G$.
The algebra $\mathcal H(\GL_n(\A_F^{\infty}),U)$ acts on $C^\infty(G(F)\backslash G(\A_F)/U,\chi)$ and this induces the Hecke action on the relative Lie algebra cohomology. 
Then this isomorphism of cohomology is compatible with the action of $\mathcal H(\GL_n(\A_F^{\infty}),U)$ on both sides.

\subsection{Regular algebraic automorphic representations}
Let $G=\GL_n$ as above.
\begin{definition}
    Let $\chi:A_G \to \C^\times$ be a continuous homomorphism.
    We write $$L^2(G(F)\backslash G(\A_F),\chi)$$ for the space of measurable functions $f:G(F)\backslash G(\A_F) \to \C$ such that $f(ag)=\chi(a)f(g)$ for all $a\in A_G$ and 
    $$\int_{G(F)\backslash G(\A_F)^1}|f(g)|^2 dg<\infty,$$
    where
    \[G(\A_F)^1:=\{g\in G(\A_F):|\det(g)|_{\A_F}=1\}\]
    and functions which agree almost everywhere are identified.
    Let $$L^2_0(G(F)\backslash G(\A_F),\chi)$$ be the subspace of cuspidal functions in $L^2(G(F)\backslash G(\A_F),\chi)$, i.e. elements $f$ of $L^2(G(F)\backslash G(\A_F),\chi)$ such that for the unipotent radical $U_P$ of every proper parabolic subgroup $P$,
    \[\int_{U_P(F)\backslash U_P(\A_F)} f(ug) du=0\] 
    for almost all $g\in G(\A_F)^1$.
    We define a \emph{cuspidal automorphic representation of character} $\chi$ to be an irreducible subrepresentation in $L^2_0(G(F)\backslash G(\A_F),\chi)$ of $G(\A_F)$.
    A \emph{cuspidal automorphic representation} is one such representation for some $\chi$.
\end{definition}

\begin{definition}
    Let $\chi:A_G \to \C^\times$ be a continuous homomorphism.
    We let $$L^2_d(G(F)\backslash G(\A_F),\chi)$$ be the discrete spectrum, i.e. the closure of the sum of all irreducible subrepresentations of $L^2(G(F)\backslash G(\A_F),\chi)$.
    We define a \emph{discrete automorphic representation of character} $\chi$ to be an irreducible subrepresentation in $L^2_d(G(F)\backslash G(\A_F),\chi)$ of $G(\A_F)$.
    A \emph{discrete automorphic representation} is one such representation for some $\chi$.
\end{definition}

\begin{remark}
    Note that every cuspidal automorphic representation is discrete. Also, by our definition, every discrete automorphic representation $\pi$ is a unitary Hilbert space representation of $G(\A_F)$ after twisting by a suitable character.\footnote{For instance, we can always twist $\pi$ by a Hecke character such that the action of $A_G$ is trivial, i.e. $\chi=1$. This reduces us to the case in \cite[section 3.7]{Getz-Hahn}.} It follows from the irreducibility of $\pi$ that $\pi$ has a central character.
    Moreover, $\pi_\infty$ is admissible by a result of Harish-Chandra \cite[Theorem 4.4.5]{Getz-Hahn}
\end{remark}

\begin{definition}
    Let $\lambda\in (\Z^n_+)^{\Hom(F,\C)}.$
    We say that a cuspidal automorphic representation $\pi$ is \emph{regular algebraic\footnote{This is $C$-algebraic in the sense of Buzzard-Gee.}/cohomological of weight $\lambda$} if $\pi_\infty$ has the same infinitesimal character as $M_{\lambda,\C}^{\vee}$. 
\end{definition}

\begin{lemma}\label{lem:inf determines character}
    Let $v$ be an infinite place of $F$.
    Let $(\rho,V)$ be an irreducible admissible representation of $G(F_v)$ that has central character $\omega$. Then the infinitesimal character of $\rho$ determines $\omega|_{F_v^{\times,\circ}}$, where $F_v^{\times,\circ}$ is the identity component of $F_v^\times$.
\end{lemma}

\begin{proof}
    First assume $v$ is a real place.
    Let us use the corresponding real embedding to identify $F_v$ with $\R$.
    Then there is $s\in\C$ such that $\omega(y)=y^s$ for all $y\in F_v^{\times,\circ}=\R_{>0}.$
    Let $X=\begin{pmatrix}
        1 & \\
        & \ddots &\\
        & & 1
    \end{pmatrix}$
    as an element in the centre of the complexified universal enveloping algebra of $G(F_v)$.
    By definition of the Lie algebra action, for all $u\in V_{sm}$,
    \[X\cdot u = \frac{d}{dt}\Big|_{t=0}\rho(e^{Xt})u=\frac{d}{dt}\Big|_{t=0} e^{st}u=su.\]
    Thus, the infinitesimal character determines $s$ and hence $\omega|_{F_v^{\times,\circ}}$.

    Now, assume $v$ is a complex place.
    Let $\sigma_1,\sigma_2:F\to \C$ be the two complex embeddings corresponding to $v$. Use the same notation for the induced isomorphisms $F_v\xrightarrow{\sim}\C.$
    Then there are $s_1,s_2\in\C$ with $s_1-s_2\in\Z$ such that $\omega(y)=\sigma_1(y)^{s_1}\sigma_2(y)^{s_2}$ for all $y\in F_v^{\times}\;(\cong\;\C^\times).$
    For $x\in F_v$, let $X=\begin{pmatrix}
        x & \\
        & \ddots &\\
        & & x
    \end{pmatrix}$
    as an element in the centre of the complexified universal enveloping algebra of $G(F_v)$.
    By definition of the Lie algebra action, for all $u\in V_{sm}$,
    \[X\cdot u = \frac{d}{dt}\Big|_{t=0}\rho(e^{Xt})u=\frac{d}{dt}\Big|_{t=0} e^{s_1\sigma_1(x)t+s_2\sigma_2(x)t}u=(s_1\sigma_1(x)+s_2\sigma_2(x))u.\]
    Taking $x=1$ and $x=\sigma_1^{-1}(i)$ gives us two equations that allows us to solve for $s_1,s_2$.
    Thus, the infinitesimal character determines $s_1,s_2$ and hence $\omega|_{F_v^{\times}}$.
\end{proof}

\begin{lemma}\label{lem: character of regular alg}
    If $\pi$ is a regular algebraic automorphic representation of weight $\lambda$, then the restriction of its central character to $F_\infty^{\times\circ}$ is the inverse of that of $M_{\lambda,\C}$.
    In particular, $A_G$ acts trivially on $\pi_\infty\otimes M_\lambda$.
\end{lemma}

\begin{proof}
    Let $M_{\lambda_v}:=\begin{cases}
        M_{\lambda_{\tau},\C} &\text{if $v$ is real}\\
        M_{\lambda_{\tau},\C}\otimes_\C M_{\lambda_{\bar\tau},\C} &\text{if $v$ is complex}
    \end{cases}$,
    where $\tau$ is an embedding $F\to \C$ whose associated place is $v$.
    Then the infinitesimal character of $\pi_v$ is the same as that of $M_{\lambda_v}^\vee$.
    We want to show that the restriction of the central character of $\pi_v$ to $F_v^{\times,\circ}$ is that of $M_{\lambda_v}^\vee=M_{\lambda_v^\vee}$.
    This follows from \Cref{lem:inf determines character}.
\end{proof}

\begin{definition}
    We define $$b_n:=r_1 \lf n^2/4\rf+r_2 n(n-1)/2,\; t_n:= r_1 \lf (n+1)^2/4\rf+r_2 n(n+1)/2-1.$$ If $n$ is clear, we may just write $b$ and $t$ instead.\footnote{It is also common in the literature to write $q_0$ for $b_n$ and $q_0+\ell_0$ for $t_n$.}
\end{definition}

\begin{definition}\label{defn: permissible}
    Let $\pi$ be a regular algebraic cuspidal automorphic representation of weight $\lambda$.
    Let $\epsilon\in (K_\infty/K_{\infty}^{\circ})^{\widehat{}}=\{1,sgn\}^{r_1}$. 
    We say that $\epsilon$ is a \emph{permissible signature} if $n$ is even, or $n$ is odd and $\epsilon_v$ is the central character of $\pi_v\otimes M_{\lambda_v}$ restricted to $\{\pm 1\}$ for all real places $v$.
\end{definition}

\begin{lemma}\label{1 dim space}
    Let $\pi$ be a regular algebraic cuspidal automorphic representation of weight $\lambda$. Let $\epsilon$ be  a permissible signature. Then for $i\in\{b_n,t_n\}$, the space 
\begin{equation*}
    H^{i}(\mathfrak g, K_\infty^0,\pi_\infty\otimes_\C M_{\lambda,\C})[\epsilon] 
\end{equation*}
is 1 dimensional, where $[\epsilon]$ denotes the $\epsilon$-isotypic component.
\end{lemma}

The strategy is to use K\"unneth formula and Clozel's result for Lie algebra cohomology of $\pi_v\otimes M_{\lambda_v}$. A slight complication is caused by the fact that our $K_\infty$ (which contains $A_G$) does not factor as a product over the infinite places.

\begin{proof}
    Let 
    \begin{align*}
        G^1(F_\infty)&:=\left\{(g_v)\in G(F_\infty): \prod_{v\mid\infty}|\det g_v|_{F_v}=1\right\},\\
        G^1(F_v)&:=\{g_v\in G(F_v): |\det g_v|_v=1\},\\
        H&:=\left\{(a_v)\in \prod_{v\mid \infty}\R_{>0}:\prod_v |a_v|_{F_v}=1\right\}\subset G(F_\infty),
    \end{align*}
    where we view $\prod_{v\mid \infty}\R_{>0}$ as a subgroup of the centre of $G(F_\infty)$ via the diagonal embedding.
    Let $\mathfrak g^1_\infty, \mathfrak g^1_v, \mathfrak h$ be the corresponding Lie algebras.
    Let $$K^1_v:=
    \begin{cases}
        O_n(\R) \text{ if $v$ is real}\\
        U_n(\R) \text{ if $v$ is complex}
    \end{cases},
    $$
    which is a subgroup of $G^1(F_v)$, so $\mathfrak k^1_v :=Lie(K^1_v) \subset \mathfrak g^1_v.$

    From the Lie group decompositions $G(F_\infty)= G^1(F_\infty) \times A_G$ and $K_\infty= \prod_{v\mid \infty}K^1_v \times A_G$, we get corresponding decompositions of Lie algebras, which in turns give
    \begin{equation}\label{eq: lie alg decomposes}
        \mathfrak g/\mathfrak k_\infty = \mathfrak g^1_\infty /\prod_{v\mid \infty}\mathfrak k^1_v
    \end{equation}
    as $\C[K_\infty]$-modules, where $K_\infty$ acts by conjugation.

    Similarly, from the Lie group decomposition $G^1(F_\infty) =\prod_{v\mid \infty}G^1(F_v) \times H$, we get a corresponding Lie algebra decomposition for $\mathfrak g^1_\infty$, which we can substitute to equation \eqref{eq: lie alg decomposes} to get
    \begin{equation}\label{eq: lie alg decomposes v2}
        \mathfrak g/\mathfrak k_\infty = \left(\prod_{v\mid \infty}\frac{\mathfrak g^1_v}{\mathfrak k^1_v}\right) \times \mathfrak h
    \end{equation}
    as $\C[K_\infty]$-modules. Note that the action of $K_\infty$ on $\mathfrak h$ is trivial.

    The relative Lie algebra complex computing $H^{i}(\mathfrak g, K_\infty^0,\pi_\infty\otimes_\C M_{\lambda,\C})$ is by definition the $i$-th cohomology of
    \begin{align*}
        &\;\;(\wedge^* (\mathfrak g/\mathfrak k_\infty)^\vee \otimes_\R (\pi_\infty)_{K_\infty\text{-fin}} \otimes_\C M_{\lambda,\C})^{K_\infty^\circ}\\
        = &\left(\wedge^* \left(\prod_{v\mid \infty}\frac{\mathfrak g^1_v}{\mathfrak k^1_v}\right)^\vee\otimes_\R \wedge^*\mathfrak h^\vee \otimes_\R (\pi_\infty)_{K_\infty\text{-fin}} \otimes_\C M_{\lambda,\C}\right)^{K_\infty^\circ}
    \end{align*}
    where $(\pi_\infty)_{K\text{-fin}}$ means the $K_\infty$-finite vectors of $\pi_\infty$.
    As the action of $K_\infty$ on $\mathfrak h$ is trivial, we can pull out that factor from the cohomology and get
    $$H^{i}(\mathfrak g, K_\infty^0,\pi_\infty\otimes_\C M_{\lambda,\C}) = \bigoplus_{a+b=i} H^a\left(\left(\wedge^* \left(\prod_{v\mid \infty}\frac{\mathfrak g^1_v}{\mathfrak k^1_v}\right)^\vee\otimes_\R (\pi_\infty)_{K_\infty\text{-fin}} \otimes_\C M_{\lambda,\C}\right)^{K_\infty^\circ}\right) \otimes_\R \wedge^b\mathfrak h^\vee .$$
    Since $A_G$ acts trivially on $\wedge^* \left(\prod_{v\mid \infty}\frac{\mathfrak g^1_v}{\mathfrak k^1_v}\right)^\vee$ and $\pi_\infty\otimes_\C M_{\lambda,\C}$ by \Cref{lem: character of regular alg}, we can replace $K_\infty^\circ$ on the right hand side by $\prod_{v\mid \infty}K_v^1$. Then by definition
    $$H^{i}(\mathfrak g, K_\infty^0,\pi_\infty\otimes_\C M_{\lambda,\C}) = \bigoplus_{a+b=i}H^{a}\left(\prod_{v\mid\infty}\mathfrak g^1_v, \prod_{v\mid\infty}\mathfrak k^1_v,\pi_\infty\otimes_\C M_{\lambda,\C}\right) \otimes_\R \wedge^b\mathfrak h^\vee,$$
    which by K\"unneth formula \cite[section 1.3 equation (2)]{BW_cts_cohomology} equals
    $$\bigoplus_{a_1+\dots+a_{m}+b=i}H^{a_1}\left(\mathfrak g^1_v,\mathfrak k^1_v,\pi_v\otimes_\C M_{\lambda_v,\C}\right)\otimes_\C\cdots\otimes_\C H^{a_m}\left(\mathfrak g^1_v,\mathfrak k^1_v,\pi_v\otimes_\C M_{\lambda_v,\C}\right) \otimes_\R \wedge^b\mathfrak h^\vee,$$
    where $m:=r_1+r_2$, $M_{\lambda_v,\C}:=\begin{cases}
        M_{\lambda_{\tau},\C} &\text{if $v$ is real}\\
        M_{\lambda_{\tau},\C}\otimes_\C M_{\lambda_{\bar\tau},\C} &\text{if $v$ is complex}
    \end{cases}$,
    and $\tau$ is an embedding $F\to \C$ whose associated place is $v$.
    The result now follows from \cite[Lemma 3.14]{Clozel}, \Cref{rmk: clozel}, and the fact that $\mathfrak h\cong \R^{r_1+r_2-1}$.
\end{proof}

\begin{remark}\label{rmk: clozel}
    For even $n$, \cite[Lemma 3.14]{Clozel} only worked with trivial $\epsilon_v$. To deduce the result for general $\epsilon$, one can twist $\pi$ by a suitable Hecke character $F^\times\backslash \A_F^\times \to \{\pm 1\}$, which can for instance be constructed from a suitable character of $\overline{F^\times F_\infty^{\times,\circ}}\backslash \A_F^\times\cong \Gal(F^{ab}/F)$.
\end{remark}

\subsection{More on cohomology}

\begin{definition}
    We define the \emph{cuspidal cohomology} as $$H^*_{cusp}(X_U,M_{\mu,\C}):=H^*(\mathfrak g,K_\infty^\circ , M_{\mu,\C}\otimes_\C L^2_0(G(F)\backslash G(\A_F),\chi)^U),$$
    where $\chi^{-1}$ is the restriction to $A_G$ of the central character of $G(F\otimes \C)$ on $M_{\mu,\C}$.
\end{definition}
The cuspidal cohomology is also acted on by $\mathcal H(G(\A_F^{\infty}),U)$ and there is a $\mathcal H(G(\A_F^{\infty}),U)$-equivariant injection
\[H^*_{cusp}(X_U,M_{\mu,\C}) \hookrightarrow H^*(X_U,M_{\mu,\C}).\]
By multiplicity one and semisimplicity of $L^2_0(G(F)\backslash G(\A_F),\chi)$ \cite[Corollary 9.1.3, Theorem 11.4.3]{Getz-Hahn}, we know 
$$L^2_0(G(F)\backslash G(\A_F),\chi)=\widehat\bigoplus_{\pi}\pi,$$ where the sum ranges over all cuspidal automorphic representations of character $\chi$.
We have a corresponding decomposition of the cuspidal cohomology into finite algebraic sum \cite[Theorem 4.1]{Schwermer}, \cite[Lemma 3.15]{Clozel}:
\begin{equation}\label{decomp of H cusp}
    H^*_{cusp}(X_U,M_{\mu,\C})= \oplus_\pi H^*(\mathfrak g,K_\infty^\circ , M_{\mu,\C}\otimes_\C \pi_\infty)\otimes_\C (\pi^\infty)^U.
\end{equation}
By strong multiplicity one, for every cuspidal automorphic representation of character $\chi$, the $(\pi^{\infty,S})^{U^S}$-isotypic component is 
\begin{equation}\label{eq: cusp isotypic}
    H^*_{cusp}(X_U,M_{\mu,\C})[(\pi^{\infty,S})^{U^S}]=H^*(\mathfrak g,K_\infty^\circ ,\pi_\infty\otimes_\C M_{\mu,\C})\otimes_\C (\pi^\infty)^U.
\end{equation}
By \Cref{1 dim space}, if $\pi$ is regular algebraic cuspidal of weight $\lambda$ with $\pi^U\neq 0$, then 
there is a $\H(\GL_n(\A_F^\infty),U)$-equivariant injection $(\pi^\infty)^U \hookrightarrow H^{b_n}_{cusp}(X_U,M_{\lambda,\C})$. The same holds for the top degree $t_n$.

\begin{definition}
    We define the \emph{inner cohomology} by $H^*_!=im(H^*_c \to H^*)$, where $H^*_c$ is the compactly supported cohomology.
\end{definition}
Then $H^*_!(X_U,M_{\mu,A})$ for $A\in\{\O,E,\C\}$ is also acted on by the Hecke algebras by restriction of the action on $H^*(X_U,M_{\mu,A})$.
Moreover, there is a Hecke-equivariant injection 
\[H^*_{cusp}(X_U,M_{\mu,\C}) \hookrightarrow H^*_!(X_U,M_{\mu,\C}).\]

\begin{lemma}\label{lem:T acts semisimply}
    \begin{enumerate}[(a)]
        \item \label{semisimple}$\mathcal H_\C:=\mathcal H(G^S,U^S)\otimes_\Z \C$ acts semisimply on $H^*_!(X_U,M_{\mu,\C})$ and there is a decomposition
        \[H^*_!(X_U,M_{\mu,\C})\cong \bigoplus_{\Pi\in \Pi_d} m(\Pi)(\Pi^{\infty,S})^{U^S},\]
        where $\Pi_d$ is the set of isomorphism classes of all discrete automorphic representations occurring as a subrepresentation of $L^2_d(\GL_n(F)\backslash \GL_n(\A_F),\chi)$ and $m(\Pi)\in \Z_{\ge 0}$.
        \item \label{isotypic} If $\pi$ is a cuspidal automorphic representation of character $\chi$, then the $(\pi^{\infty,S})^{U^S}$-isotypic component is 
        \[H^*_!(X_U,M_{\mu,\C})[(\pi^{\infty,S})^{U^S}]=H^*(\mathfrak g,K_\infty^\circ,\pi_\infty\otimes_\C M_{\mu,\C})\otimes_\C (\pi^\infty)^U.\]
        \item \label{reduced}If $U^S:=\GL_n(\prod_{v\notin S\cup\{\infty\}}\O_{F_v})$, then the ring $$\T^S_\C(H^*_!(X_U,M_{\mu,\C})):=im(\mathcal H_\C \to \End_\C(H^*_!(X_U,M_{\mu,\C})))$$ is reduced.
    \end{enumerate}
\end{lemma}

\begin{proof}
    Let $H_2^*=H^*(\mathfrak g,K_\infty^\circ,L^2_d(\GL_n(F)\backslash \GL_n(\A_F),\chi)_{sm}\otimes_\C M_{\mu,\C}),$ where $()_{sm}$ stands for the smooth vectors.
    We have a Hilbert space decomposition
    \[L^2_d(\GL_n(F)\backslash \GL_n(\A_F),\chi)\cong\widehat\bigoplus_{\Pi\in \Pi_d}m_d(\Pi)\Pi.\]
    By the multiplicity one theorem for the discrete spectrum proved by Moeglin and Waldspurger \cite{Moeglin-Waldspurger}, we have $m_d(\Pi)=1$ for all such $\Pi$.
    We have a corresponding decomposition for the cohomology \cite[Theorem 5.3]{Borel-Garland}:
    \begin{equation}\label{eq:H_2 decomp}
        H_2^*\cong \bigoplus_{\Pi\in \Pi_d} H^*(\mathfrak g,K_\infty^\circ,\Pi_\infty\otimes_\C M_{\mu,\C})\otimes_\C (\Pi^\infty)^U.
    \end{equation}
    Each $\Pi^\infty$ is an irreducible admissible representation of $\GL_n(\A_F^\infty)$, so it is factorisable, so $(\Pi^\infty)^U=(\Pi_S)^{U_S}\otimes_\C (\Pi^{\infty,S})^{U^S}$ is isomorphic to a direct sum of simple modules of $\mathcal H_\C$.
    It follows that $H_2^*$ is a semisimple $\mathcal H_\C$-module.

    By a result of Borel \cite[Proposition 3.18]{Clozel}, we have injections
    \begin{equation}\label{eq:3 injections}
        H^*_{cusp}(X_U,M_{\mu,\C}) \hookrightarrow H^*_{!}(X_U,M_{\mu,\C}) \hookrightarrow \tilde H^*_2
    \end{equation}
    where $\tilde H^*_2$ is the image of $H^*_2\to H^*(X_U,M_{\mu,\C}).$ 
    Submodules of semisimple modules are semisimple, so $H^*_{!}(X_U,M_{\mu,\C})$ is a semisimple $\mathcal H_\C$-module and part \ref{semisimple} follows.
    
    For part \ref{isotypic}, recall that discrete automorphic representations are isobaric (as they are Speh representations), so they satisfy strong multiplicity one, so $H_2^*[(\pi^{\infty,S})^{U^S}]=H^*(\mathfrak g,K_\infty^\circ,\pi_\infty\otimes_\C M_{\mu,\C})\otimes_\C (\pi^\infty)^U$ by \eqref{eq:H_2 decomp}.
    But this is also the $(\pi^{\infty,S})^{U^S}$-isotypic component of the cuspidal cohomology by equation \eqref{eq: cusp isotypic}.
    Part \ref{isotypic} now follows from \eqref{eq:3 injections}.

    For part \ref{reduced}, assume $U^S:=\GL_n(\prod_{v\notin S\cup\{\infty\}}\O_{F_v})$. By Satake isomorphism, $\mathcal H_\C$ is commutative. Thus, $(\Pi^{\infty,S})^{U^S}$ is $0$-dimensional or $1$-dimensional. Hence each element of $\mathcal H_\C$ acts on $(\Pi^{\infty,S})^{U^S}$ by a scalar. The result now follows from part \ref{semisimple}.
\end{proof}

\section{Abstract congruence ideals}\label{sec: congruence module}
In this section, we will define congruence ideals as in \cite[section 2.1]{tu}. We will also establish some of their properties in the abstract algebraic setting. These will be applied to the Hecke algebras and cohomology of locally symmetric space in the next section.

Let $\mathcal O$ be a complete discrete valuation ring with uniformizer $\varpi$ and field of fractions $E$.
Let $T$ be a reduced finite flat local $\mathcal O$-algebra, $\lambda:T\to \mathcal O$ be an $\mathcal O$-algebra homomorphism.
This, being a section of the structure map $\mathcal O\to T$, is necessarily surjective.
We first recall a standard concept, which already appeared in \cite{wiles_modular_1995}:
\begin{definition}\label{defn: Wiles defn of eta}
    $\eta_\lambda:=\lambda(Ann_T (\ker\lambda))$.
\end{definition}
It turns out that to study $\eta_\lambda$, it is useful to generalise this concept to modules over $T$.

Let $M$ be a finitely generated $T$-module which is free over $\mathcal O$.
Write $M_E=M\otimes_{\mathcal O}E$ and $T_E=T\otimes_{\mathcal O}E$.
Note that $M\hookrightarrow M_E$ and $T\hookrightarrow T_E$ by $\O$-flatness of $T$ and $M$.
Also, $\lambda$ induces an $E$-algebra map $\lambda_E:T_E\to E$.
Note that $T_E$ is a finite dimensional $E$-vector space, so it is Artinian.
It follows that 
\begin{align}
    T_E&\cong \prod_{\p\in Spec T_E} (T_E)_{\p} \label{eq:str}\\
    &\cong \prod_{\p\in Spec T_E} T_E/\p \label{eq:loc}\\
    &\cong E\times \prod_{\p\neq ker \lambda_E} T_E/\p\label{eq:isom}
\end{align}
Here, \eqref{eq:str} is given by the diagonal map;
\eqref{eq:loc} is true as $T_E$ is reduced;
\eqref{eq:isom} is true by the first isomorphism theorem.
The upshot is that we have a canonical decomposition
\[T_E\cong E\times T_E^c\]
of $E$-algebras, where the first projection is given by $\lambda_E$.

Let $e_\lambda=e\in T_E$ be the element corresponding to $(1,0)\in E\times T_E^c$.
In other words, $e$ is the unique element of $T_E$ such that $\lambda_E(e)=1$ and $e\in \bigcap_{\p\neq \ker\lambda_E}\p$.
Define two $T$-submodules of $M_E$ by
\[M^\lambda:=eM\]
and
\[M_\lambda:=eM\cap M = M[\ker\lambda]\]
where the last equality is proved in lemma \ref{lem: M_lambda}.
In \cite[section 2.1]{tu}, $M_\lambda$ is defined as $eM_E\cap M$, but it is equivalent to our definition, because if $m\in eM_E\cap M$, then $m=em\in M\cap eM$.
\begin{definition}\label{defn: congruence module}
    Define the \emph{congruence module} $C_0^\lambda(M)$ by 
    \[C_0^\lambda(M):=\frac{M^\lambda}{M_\lambda}\]
    and the \emph{congruence ideal} to be its Fitting ideal
    \[\eta_\lambda(M):=Fitt_{\mathcal O}(C_0^\lambda(M)).\]
\end{definition}

Note that $C_0^\lambda(M)$ is a finite torsion $\O$-module, so $\eta^\lambda(M)$ is completely determined by the cardinality of $C_0^\lambda(M)$. More precisely, if $C_0^\lambda(M)$ has cardinality $|\O/\varpi|^a$, then $\eta_\lambda(M)=(\varpi^a)$.

To see how definitions \ref{defn: Wiles defn of eta}, \ref{defn: congruence module} are related, note:
\begin{lemma}\label{lem: M_lambda}
    Let $M$ and $T$ be as above.
    \begin{enumerate}
        \item \label{M_lambda}$M_\lambda=M[\ker\lambda]:=\{m\in M:(\ker\lambda) m=0\}$.
        \item \label{wiles and tu}$\eta_\lambda(T)=\eta_\lambda$.
    \end{enumerate}
\end{lemma}

\begin{proof}
    For \eqref{M_lambda}, we first consider a slightly more general setup.
    Suppose we have a product of commutative rings $A\times B$ acting on $N$.
    We can decompose $N$ into $N_1\times N_2$ accordingly.
    It is then clear that $N_1=N[\ker \pi_A]$, where $\pi_A:A\times B\to A$ is the first projection.
    
    In our case, taking $A\times B=E\times T_E^c$ and $N=M_E$ shows $e M_E=M_E[\ker\lambda_E]$.
    Hence $M_\lambda=e M_E \cap M=\{m\in M : m\otimes 1\in M_E \text{ is annihilated by }(\ker\lambda)\otimes_\O E\}=M[\ker\lambda]$ since $M$ is a free $\O$-module.

    For \eqref{wiles and tu}, note that under the $\O$-module isomorphism $T_E\cong E\times T^c_E$, the $\O$-module $T^\lambda=eT$ corresponds to $\lambda(T)\times 0=\O\times 0$ while $T_\lambda=T[\ker\lambda]$ corresponds to $\lambda(T[\ker\lambda])\times 0$.
    Hence, $C_0^\lambda(T)=\mathcal O/\lambda(T[\ker\lambda])$ and $\eta_\lambda(M)=\lambda(T[\ker\lambda])=\eta_\lambda$.
\end{proof}

Observe that $M^\lambda=eM \subset M_E$ is torsion free and finitely generated over $\O$, so it is a free $\O$-module.
The same is true for $M_\lambda=M[\ker\lambda]$.
\begin{lemma}\label{lem: M lambda rk 1}(cf. \cite[equation (2.2)]{tu})
    If $rk_\O M_\lambda=1$,
    then $C_0^\lambda(M)=\mathcal O/\eta_\lambda(M)$ and $\eta_\lambda(M)\supset \eta_\lambda$. 
\end{lemma}

\begin{proof}
    We know that $M^\lambda/M_\lambda$ is a finite torsion $\O$-module, so $rk_\O (M^\lambda)=rk_\O (M_\lambda)=1$.
    Thus, there exists $m\in M$ such that $eM=\O em$ as $\O$-module.
    Then we have a surjection of $\O$-modules
    \begin{align*}
        C_0^\lambda(T)=\frac{eT}{T[\ker\lambda]} &\twoheadrightarrow C_0^\lambda (M)=\frac{eM}{M[\ker\lambda]}\\
        x&\mapsto xm.
    \end{align*}
    The lemma now follows from the observation that $C_0^\lambda(T)=\O / \eta_\lambda$.
\end{proof}

\begin{lemma}\label{lem: pp}
    Let $T,\lambda, e$ be as above.
    Let $\tilde T$ be a finite flat local $\O$-algebra, $\tilde\lambda:\tilde T\to \O$ be an $\O$-algebra homomorphism, and $\tilde e$ be the corresponding idempotent in $\tilde T$.
    Let $M_1,M_2$ be $T$-module and $\tilde T$-module respectively that are finite free over $\O$.
    Suppose that there is an $\O$-bilinear perfect pairing\footnote{This means the two induced maps $M_1\to \Hom_\O(M_2,\O)$ and $M_1\to \Hom_\O(M_2,\O)$ are isomorphisms.}
    \[[\,,\,]:M_1\times M_2 \to \O\]
    such that\footnote{We also write $[\,,\,]$ for its extension $(M_1)_E\times (M_2)_E \to E$.}
    $[eM_1,(1-\tilde e)M_2]=0$ and $[(1-e)M_1,\tilde eM_2]=0$.
    \begin{enumerate}[(a)]
        \item \label{E/O}Then $[\,,\,]$ induces an $\O$-bilinear perfect pairing 
        \[C_0^\lambda(M_1)\times C_0^{\tilde\lambda}(M_2) \to E/\O\]
        and $\eta_\lambda(M_1)=\eta_{\tilde\lambda}(M_2)$.
        \item \label{pair basis} If $M_1[\ker\lambda]$ and $M_2[\ker\tilde\lambda]$ are both free $\O$-modules of rank 1 with respective bases $\delta_1,\delta_2$, then 
        \[\eta_\lambda(M_1)=\eta_{\tilde\lambda}(M_2)=([\delta_1,\delta_2]).\]
    \end{enumerate}
\end{lemma}
The key observation is that for every finite torsion $\O$-module $N$, we have a (non-canonical) isomorphism $N\cong \Hom_\O(N,E/\O)$.
Hence, by considering the cardinalities of $C_0^\lambda(M_1)$ and $C_0^{\tilde\lambda}(M_2)$, we know the pairing in \ref{E/O} is perfect if and only if it is non-degenerate.

This lemma appears similar to \cite[Proposition 2.3]{tu}, but a key difference is that we do not assume $[tx,y]=[x,ty]$ for all $t\in T$. Instead, our analogous conditions are $[eM_1,(1-\tilde e)M_2]=0$ and $[(1-e)M_1,\tilde eM_2]=0$. This distinction is important because we will later apply this lemma to the cup product, which satisfies our conditions but not theirs. In the $\GL_2$ setting, they twisted the pairing by the Atkin-Lehner involution to make their conditions hold, but we are unaware of any such involution for $\GL_n$.

\begin{proof}
    It is easy to see that $[\,,\,]$ extends to an $E$-bilinear perfect pairing $[\,,\,]:(M_1)_E\times (M_2)_E \to E$.
    Let $em_1\in eM_1$ (with $m_1\in M_1$) and $\tilde em_2\in M_2\cap \tilde eM_2$ (with $m_2\in M_2$).
    Then $[em_1,\tilde em_2]=[em_1+(1-e)m_1,\tilde em_2]=[m_1,\tilde em_2]\in \O$ because $(m_1,\tilde em_2)\in M_1\times M_2$.
    A symmetric consideration shows $[\,,\,]$ induces a map
    \begin{equation*}
        \langle\,,\,\rangle:C_0^\lambda(M_1)\times C_0^{\tilde\lambda}(M_2) \to E/\O.
    \end{equation*}

    To show $\langle\,,\,\rangle$ is non-degenerate, we let $m_1\in M_1$. 
    Suppose that $\langle em_1,-\rangle\in \Hom_\O(C_0^{\tilde\lambda}(M_2),E/\O)$ is zero.
    This means that for all $m_2\in M_2$, 
    \[[em_1,\tilde em_2]=[em_1,\tilde em_2+(1-\tilde e)m_2]=[em_1,m_2]\in\O\]
    so $[em_1,-]\in \Hom_\O(M_2,\O)$.
    By perfectness, there exists $n\in M_1$ such that $[em_1,-]=[n,-]$, so $n=em_1\in eM_1\cap M_1=(M_1)_\lambda$ by perfectness again, so $em_1=0\in C_0^\lambda(M_1)$.
    By symmetry, $\langle\, , \, \rangle$ is non-degenerate and hence perfect.
    Also, $C_0^\lambda(M_1)\cong \Hom_\O(C_0^{\tilde\lambda}(M_2),E/\O)\cong C_0^{\tilde\lambda}(M_2).$

    For part \ref{pair basis}, let $(\varpi^a):=\eta_\lambda(M_1)=\eta_{\tilde\lambda}(M_2)$.
    By lemma \ref{lem: M lambda rk 1}, $C_0^\lambda(M_1)$ and  $C_0^{\tilde\lambda}(M_2)$ are free $\O/(\varpi^a)$-modules of rank 1, with respective bases $b_1,b_2$ say.
    By part \ref{E/O}, we have isomorphism
    \begin{align*}
        \frac{\O}{(\varpi^a)} \cong \frac{\O}{(\varpi^a)}b_1\xrightarrow{\sim} \Hom_\O\left(\frac{\O}{(\varpi^a)}b_2,E/\O\right)&\xrightarrow{\sim}\frac{\varpi^{-a}\O}{\O}\xrightarrow{\cdot\varpi^a}\frac{\O}{(\varpi^a)}\\
        f&\mapsto f(b_2).
    \end{align*}
    This isomorphism maps $1$ to $[b_1,b_2]\varpi^a \pmod {\varpi^a}$, so $[b_1,b_2]\varpi^a\in \O^\times$.
    As $\O b_i/\O\delta_i \cong \O/(\varpi^a)$, we know $\delta_i\in \varpi^a b_i\O^\times$. We deduce that $[\delta_1,\delta_2]\in \varpi^a\O^\times$, as desired.
\end{proof}

The following lemma explains why $\eta_\lambda$ is called the congruence ideal:
\begin{lemma}\label{lem:congruence}
    Let $E$ be a non-Archimedean local field with ring of integers $\O$, $T$ be a reduced finite flat local $\O$-algebra, $\lambda:T\to \O$ be an $\O$-algebra homomorphism. Then $\eta_\lambda\neq \O$ if and only if there is a finite field extension $L$ of $E$ and an $\O$-algebra homomorphism $\lambda':T\to \O_{L}$ such that (viewing $\lambda$ as a homomorphism to $\O_L$) we have 
    $\lambda\neq \lambda'$ and
    \[\lambda\equiv \lambda' \pmod{\varpi_L}.\]
    Here $\O_L$ is the ring of integers of $L$ and $\varpi_L$ is a uniformizer of $\O_L$.
\end{lemma}

\begin{remark}\label{remark: congruence}
    Recall that $T$ is finite over $\O$, so any $\O$-algebra homomorphism $T\to \overline E$ has image in $\O_L$ for some finite field extension $L$ of $E$. Thus, we can rephrase the lemma: $\eta_\lambda\neq \O$ if and only if there is an $\overline E$-algebra homomorphism $\lambda':T\otimes_{\O}\overline{E} \to \overline E$ such that $\lambda'\neq \lambda\otimes_{\O}\overline{E}$ and $|\lambda(t)-\lambda'(t)|<1$ for all $t\in T$, where $|\cdot|$ is an absolute value on $\overline E$ extending that of $E$.
\end{remark}

\begin{proof}
    Given $\lambda$, we can decompose $T\otimes_\O E\cong E\times T_E^c$ and get an idempotent $e\in T\otimes_\O E$ as before.
    Define $T^c:=im(T\to T\otimes_\O E\to T_E^c).$

    We assume that $\lambda'$ is as in the statement of the lemma.
    We claim that $\lambda'$ factors through $T\to T^c$.
    As $\O_L\subset \O_L\otimes_\O E$, it suffices to show 
    \[\lambda'_E:=\lambda'\otimes E: T\otimes_\O E\to \O_L\otimes_\O E\] factors through $T^c_E$.
    Note that $\O_L\otimes_\O E=L$ is an integral domain while $T\otimes_\O E\cong E\times T_E^c$, so $\lambda'_E$ must factor through $E$ or $T_E^c$.
    Since $\lambda'_E$ is an $E$-algebra homomorphism and $\lambda'_E \neq \lambda_E$, $\lambda'_E$ cannot factor through $E$. Hence $\lambda'_E$ must factor through $T_E^c$, as claimed. 
    Suppose $\eta_\lambda=\O$.
    This means $eT\cap T = eT$, i.e. $T\supset eT$.
    In particular, $(1,0)=e\cdot (1,1)\in T$.
    Then
    \begin{align*}
        1&=\lambda(1,0)\\
        &\equiv\lambda'(1,0)\pmod{\varpi_L}\\
        &\equiv 0\pmod\varpi
    \end{align*}   
    where the last equality holds since $\lambda'$ factors through $T^c$.
    We get a contradiction, so $\eta_\lambda\neq \O$.

    Conversely, suppose $\eta_\lambda\neq \O$.
    The key observation is that $\O/\eta_\lambda=\O\otimes_T T^c$, because $\ker(T\twoheadrightarrow T^c)=T\cap eT$.
    Thus we have an $\O$-algebra homomorphism
    $$f:T^c\twoheadrightarrow \O\otimes_T T^c=\O/\eta_\lambda \twoheadrightarrow \O/\varpi.$$
    We get a commutative diagram
\[\begin{tikzcd}
	& {\mathcal O} \\
	T && {\mathcal O/\varpi} \\
	& {T^c}
	\arrow[two heads, from=1-2, to=2-3]
	\arrow[two heads, from=2-1, to=1-2]
	\arrow[two heads, from=2-1, to=3-2]
	\arrow["f"', two heads, from=3-2, to=2-3]
\end{tikzcd}\]
We want to lift $f$.
There is a classical argument due to Deligne and Serre.
Let $\m^c:=\ker f$.
Consider the structure map $\O\to T^c$.
As $\varpi \in \m^c$, the prime ideal $\m^c$ lies above $\varpi\O$.
Note that multiplication by $\varpi$ is invertible in $T_E$, so $T_E$ and hence $T^c$ is $\O$-torsion free.
Thus, $T^c$ is a finite free $\O$-module.
By flatness, the going down property holds, so there is a prime ideal $\p^c\subset \m^c$ lying above $(0)$, so $\O\hookrightarrow T^c/\p^c$.
This is a finite extension, so $$L:= Frac(T^c/\p^c)$$ is a finite extension of $E$.
We know $T^c/\p^c$ is finite over $\O$ and hence integral over $\O$, so $$T^c/\p^c\subset \O_L.$$
We want to show that this inclusion is a local homomorphism of local rings.
As $\O$ is a Henselian ring and $T^c/\p^c$ is an integral domain, $T^c/\p^c$ is a local ring. 
Note $\m':=\varpi_L\O_L \cap T^c/\p^c$ is a prime ideal of $T^c/\p^c$ lying above $\varpi\O$ (since its pullback to $\O$ is the preimage of $\varpi_L \O_L$ under the $\O$-algebra map $\O\to T^c/\p^c \to \O_L$).
By the incomparability theorem for injective integral ring extensions, $\m'$ is a maximal ideal of $T^c/\p^c$.
Hence, the fact that $T^c/\p^c$ is a local ring implies $\m'=\m^c/\p^c$.
This gives us another commutative diagram
\[\begin{tikzcd}
	& {T^c/\mathfrak m^c} & {\mathcal O_L/\varpi_L} \\
	{T^c} & {T^c/\mathfrak p^c} & {\mathcal O_L}
	\arrow[from=1-2, to=1-3]
	\arrow["f", two heads, from=2-1, to=1-2]
	\arrow[two heads, from=2-1, to=2-2]
	\arrow[two heads, from=2-2, to=1-2]
	\arrow[hook, from=2-2, to=2-3]
	\arrow[two heads, from=2-3, to=1-3]
\end{tikzcd}\]
Combining the two diagrams gives us
\[\begin{tikzcd}
	& {\mathcal O} \\
	T &&& {\mathcal O_L/\varpi_L} \\
	& {T^c} & {\mathcal O_L}
	\arrow[from=1-2, to=2-4]
	\arrow["\lambda", from=2-1, to=1-2]
	\arrow[from=2-1, to=3-2]
	\arrow[from=3-2, to=3-3]
	\arrow[from=3-3, to=2-4]
\end{tikzcd}\]
Denote the bottom map $T\to\O_L$ by $\lambda'$.
It remains to show $\lambda \neq \lambda'$ as maps to $\O_L$.
If this is false, then for all $(x,y)\in T\subset \O\times T^c$, we have $x-y\in \p^c$ by construction of $\lambda'$.
In particular, if $(x,y)\in T\cap eT$, then $y=0$ and $x\in \O\cap \p^c=(0)$ as $\p^c$ lies above $(0)$.
This means $T\cap eT=0$, which is false as $\varpi^a e\in T\cap eT$ for $a\in \Z$ sufficiently large.
\end{proof}

\section{\texorpdfstring{$L(1,\pi,\Ad^\circ)$}{L(1,pi,Ad 0)} and congruences for automorphic representations}\label{sec:4}
In this section, we will apply the results from the last section to study congruence ideals of Hecke algebras and cohomology of locally symmetric spaces.
Then we will relate them to Selmer groups.

We shall need some notation and assumptions.
Let
\begin{itemize}\label{notations...}
    \item $F$ be a number field
    \item $p$ be a prime. Starting from section \ref{sec: cong ideals for aut reps}, we will also assume $$p>2.$$
    \item $\iota:\overline\Q_p\xrightarrow{\sim} \C$ be a fixed isomorphism. We will often use it implicitly.
    \item $\pi$ be a cuspidal, regular algebraic automorphic representation of $\GL_n(\A_F)$ of weight $\iota\mu$, where $\mu\in(\Z^n)^{\Hom(F,\overline\Q_p)}$
    \item $U_{v,m}=\{g\in \GL_n(\O_{F_v}): \text{the last row of } g\pmod{\varpi_v^{m}} \text{ is }\begin{pmatrix}
        0 & \dots & 0 & 1
    \end{pmatrix}
    \}$ for $m\in\Z_{\ge 0}$, where $\varpi_v$ is a uniformiser of $F_v$.
    Let $f_v\in \Z_{\ge 0}$ be the unique integer such that\footnote{Such a $f_v$ always exists by the uniqueness of local new vector \cite[Theorem 11.5.6]{Getz-Hahn}.} $$\dim_\C (\pi_v)^{U_{v,f_v}} =1.$$
    Let
    $U=\prod_{v\nmid\infty}U_v$, where $U_v=U_{v,f_v}$.
    Assume that $U$ is \textbf{neat}.
    \item $S$ be a finite set of finite places of $F$ containing all $v$ such that $\pi_v$ is ramified.
    \item $G^S=\GL_n(\A_F^{S,\infty})$ and $U^S:=\GL_n(\prod_{v\notin S\cup\{\infty\}}\O_{F_v})$.
    \item $E\subset \overline \Q_p$ be a local field containing $\iota^{-1}(\Q(\pi))$ \footnote{$\Q(\pi)$ is the field of rationality of $\pi$ \cite[section 3 page 101]{Clozel}.}and the image of every embedding $F\hookrightarrow \overline\Q_p$
    \item $\O$ be the ring of integers of $E$, $\varpi$ a uniformizer of $\O$
    \item $\epsilon\in \widehat{K_\infty/K_\infty^\circ}$ a permissible signature (\Cref{defn: permissible})
    \item $X_U$ the locally symmetric space of $\GL_{n,F}$ (\Cref{lss}), $\partial X_U$ the boundary of its Borel-Serre compactification.
    \item $\mathbb T^S:= \mathcal H(G^S,U^S)\otimes_\Z \O.$
\end{itemize}

Note that $\mathbb T^S$ is a commutative $\O$-algebra.
If $M$ is an $\O$-module equipped with an $\O$-algebra homomorphism $\mathbb T^S\to \End_\O(M)$, then we define 
\[\mathbb T^S(M):=im(\mathbb T^S\to \End_\O(M)).\]

For completeness, let us also remark that in \Cref{lem:Lambda} below, we shall show that there is a Hecke eigensystem $\Lambda:\T^S\to\O$ attached to $\pi$. We will let $$\m:=\ker(\Lambda \pmod\varpi).$$



\subsection{Hecke eigensystems}
We first show that we have a Hecke eigensystem attached to $\pi$.
\begin{lemma}\label{lem:Lambda}
    We have an $\O$-algebra homomorphism
    \begin{equation*}
        \Lambda:\mathbb T^S\to \O
    \end{equation*}
    sending $t\in \mathbb T^S$ to its eigenvalue on $(\iota^{-1}\pi^{\infty})^{U}$.
    This homomorphism factors through $\mathbb T^S(H^*_!(X_U,M_{\mu,\O})),$ where $H^*_!=im(H_c^*\to H^*)$ is the inner cohomology and $H^*=\oplus_{i\ge 0}H^i$.
\end{lemma}

\begin{proof}
    In this proof, we may sometime abuse notation and regard a $\C$-vector space as an $\O$-module via the map $\iota:\overline\Q_p\xrightarrow{\sim} \C$.
    Let $\mathcal H:=\mathcal H(G^S,U^S)$.

    As $\GL_n(\A_F^S)$ acts on the $\overline\Q_p$-vector space $\iota^{-1}\pi^\infty$, we know $\mathbb T^S$ acts on the $U^S$-invariant $(\iota^{-1}\pi^\infty)^{U^S}$ and hence also on $(\iota^{-1}\pi^\infty)^{U}$.
    Moreover, for all finite $v\notin S$, $\pi_v$ is unramified so $\pi_v^{U_v}$ is a one dimensional $\overline\Q_p$-vector space.
    It follows that each element of $\mathbb T^S$ acts by a scalar on $(\iota^{-1}\pi^\infty)^{U}$, so we get an $\O$-algebra homomorphism
    \begin{equation}\label{eq:eigenvalue}
        \mathbb T^S \to \overline \Q_p
    \end{equation}
    sending an element to its eigenvalue.

    Note that $\mathbb T^S(H^*_!(X_U,M_{\mu,\O}))$ acts on $H^*_!(X_U,M_{\mu,\O})\otimes_{\O}\C \cong H^*_!(X_U,M_{\mu,\C})$.
    Since $\pi$ is cohomological of weight $\mu$, we have $\mathcal H$-equivariant injections 
    \[(\pi^\infty)^U \hookrightarrow H^*_{cusp}(X_U,M_{\mu,\C})\hookrightarrow H^*_!(X_U,M_{\mu,\C}).\]
    Pick any non-zero $x\in (\pi^\infty)^U$ and let $y$ be its image under this injection.
    For all $t\in \mathbb T^U$, its eigenvalue on $(\iota^{-1}\pi^\infty)^U$ only depends on how it acts on $y$ and this is determined by the image of $t$ in $\mathbb T^S(H^*_!(X_U,M_{\mu,\O}))$. 
    It follows that \eqref{eq:eigenvalue} factors through $\mathbb T^S(H^*_!(X_U,M_{\mu,\O})).$  
    It remains to show that the image of \eqref{eq:eigenvalue} lies in $\O$.

    We first show that its image lies in $E$.
    By \cite[Proposition 3.1]{Clozel}, there is a $\GL_n(\A_F^\infty)$-stable $\Q(\pi)$-vector subspace $W$ of $\pi^\infty$ such that $\pi^\infty=W\otimes_{\Q(\pi)}\C.$
    Then $(\pi^\infty)^U=W^U\otimes_{\Q(\pi)}\C.$
    Let $h\in \mathcal H$.
    We have already seen that it acts by a scalar on $(\pi^\infty)^U$, so the same is true for $W^U$.
    As $W^U$ is a $\Q(\pi)$-vector space, the scalar must lies in $\Q(\pi)$.
    Hence the image of \eqref{eq:eigenvalue} lies in $\iota^{-1}(\Q(\pi))\subset E$.

    Finally, by the existence of Borel-Serre compactification of $X_U$, we know that $H^*_!(X_U,M_{\mu,\O})$ is a finite $\O$-module, so 
    \[H^*_!(X_U,M_{\mu,\O})/\O\text{-torsion}\]
    is a finite free $\O$-module stable under $\mathcal H$.
    Pick an $\O$-basis $\mathcal B$ for this module.
    We can then express the action of $\mathcal H$ on $H^*_!(X_U,M_{\mu,\O})/\O\text{-torsion}$ by matrices with entries in $\O$.
    We can view $\mathcal B$ as a $\C$-basis for $H^*_!(X_U,M_{\mu,\C})$.
    Then the action of $\mathcal H$ on $H^*_!(X_U,M_{\mu,\C})$ is given by the same matrices.
    Hence each eigenvalue of $\mathcal H$ on this space is a root of a monic polynomial over $\O$.
    We know they lie in $E$ by the previous paragraph, so they lie in $\O$ as $\O$ is integrally closed.
\end{proof}

We let $\mathfrak m=\ker(\Lambda \mod \varpi)\subset \mathbb T^S$, which is a maximal ideal. We define
\[T:=\mathbb T^S(H^*_!(X_U,M_{\mu,\O}))_\m/\O\text{-torsion}.\]
\begin{lemma}\label{lem: small lambda}
    $T$ is a reduced finite flat\footnote{Note that since $\O$ is a PID, an $\O$-module is finite flat if and only if it is finite free.} local $\O$-algebra. Also, $\Lambda$ induces a local $\O$-algebra map
    \[\lambda:T\to \O.\]
\end{lemma}

\begin{proof}
    Let $H:=H^*_!(X_U,M_{\mu,\O})$.
    By the existence of Borel-Serre compactification of $X_U$, $H$ is a finite $\O$-module, so $\T^S(H)$ is a finite $\O$-algebra.
    We let $$\Lambda':\mathbb T^S(H) \to \O$$ be the map induced by $\Lambda$, which exists by lemma \ref{lem:Lambda},
    and let $q:\T^S\to \T^S(H)$ be the quotient map, so $\Lambda=\Lambda'\circ q$.
    Then $\m\supset \ker q$, so $q(\m)$ is a maximal ideal of $\T^S(H)$ and hence
    \[\T^S(H)_\m= \T^S(H)_{q(\m)}\]
    is a finite local $\O$-algebra. The same is thus true for $T$.
    As $T$ is a finitely generated torsion-free $\O$-module and $\O$ is a PID, $T$ is flat over $\O$.
    Hence $T\hookrightarrow T\otimes_\O \C$.
    As $\C$ is $\O$-flat, $T\otimes_\O \C=\T^S_\C(H^*_!(X_U,M_{\mu,\C}))_\m,$ which is reduced by \Cref{lem:T acts semisimply}. Thus, $T$, which injects into $T\otimes_\O \C$, is also reduced.

    Note that $q(\m)=\ker(\Lambda'\mod \varpi)$. It follows that $\Lambda'$ induces an $\O$-algebra map $\T^S(H^*_!(X_U,M_{\mu,\O}))_{q(\m)}\to \O$ and hence an $\O$-algebra map $\lambda:T\to \O$, i.e. a commutative diagram
\[\begin{tikzcd}
	T & \O \\
	\O
	\arrow["\lambda", from=1-1, to=1-2]
	\arrow[hook, from=2-1, to=1-1]
	\arrow[from=2-1, to=1-2]
\end{tikzcd}\]
 It follows that $\lambda^{-1}(\varpi \O)$ is a prime ideal of $T$ lying above $\varpi\O$ (i.e. the preimage of $\lambda^{-1}(\varpi \O)$ along the structure map $\O\to T$ is $\varpi\O$).
    Since $T$ is a local $\O$-algebra, the maximal ideal of $T$ also lies above $\varpi\O$.
    By the incomparability theorem for injective integral ring extensions, we deduce that $\lambda^{-1}(\varpi \O)$ is the maximal ideal of $T$. Thus, $\lambda$ is a local homomorphism.
\end{proof}

\begin{remark}\label{tem}
    Let us show that $$T\cong\T^S(\overline H^*_!(X_U,M_\mu)_\m),$$ where $\overline H^*_!(X_U,M_\mu):=H^*_!(X_U,M_\mu)/\O\text{-torsion}.$ (In the following, a bar on top of an $\O$-module will usually mean the module modulo its $\O$-torsion.)
    Write $H=H^*_!(X_U,M_{\mu,\O}).$
    There is an obvious surjection 
    \begin{equation}\label{eq:tem}
        \T^S(H) \to \T^S((\overline H)_\m).
    \end{equation}
    Let $t\in \T^S(H)$.
    Since $H$ is a finitely generated $\O$-module and $(\overline H)_\m$ is a quotient of it, we know that $t$ is in the kernel of \eqref{eq:tem} if and only if there exists $a\in\mathbb N, b\in \T^S-\m$ such that $\varpi^a bt=0$, which is equivalent to $t \in \ker(\T^S(H)\twoheadrightarrow \T^S(H)_\m/\O\text{-tors})$. 
    Thus \eqref{eq:tem} induces the desired isomorphism.
\end{remark}

\subsection{Congruence ideals for automorphic representations}\label{sec: cong ideals for aut reps}
The previous lemma means we are now in the situation of section \ref{sec: congruence module}.
\begin{definition}\label{defn: congruence ideal}
    With the setup above, define the \emph{congruence ideal}
    \[\eta_{\pi}:=\eta_\lambda=\lambda(Ann_T(\ker\lambda))=Fitt_\O\left(\frac{eT}{eT\cap T}\right)\] 
    as in definitions \ref{defn: Wiles defn of eta}, \ref{defn: congruence module}, where $e$ is the idempotent in $T\otimes_\O E$ corresponding to $(1,0)$ in the decomposition $T\otimes_\O E \cong E \times T^c_E$ induced by $\lambda$.
    For $i\in\{b, t\}$ and a permissible $\epsilon\in \widehat{K_\infty/K_\infty^\circ}$, define the \emph{cohomological congruence ideal}  
    \[\eta_{\pi,i,\epsilon}:=\eta_{\lambda}(H^i_!(X_U,M_{\mu,\O})_\m[\epsilon]/\O\text{-torsion}) = Fitt_\O\left(\frac{eH}{eH\cap H}\right)\]
    where $H=H^i_!(X_U,M_{\mu,\O})_\m[\epsilon]/\O\text{-torsion}$.
\end{definition}

Let $\O_{\overline\Q_p}$ be the ring of integers of $\overline\Q_p$ and $\m_{\overline\Q_p}$ be its maximal ideal.
\begin{lemma}\label{lem: M ker lambda}
    Fix $i\in\{b, t\}$ and a permissible $\epsilon\in \widehat{K_\infty/K_\infty^\circ}$.
    Let $p>2$ and $\T^S_{\overline\Q_p} := \T^S\otimes_{\O} \overline\Q_p$.
    \begin{enumerate}[(a)]
        \item\label{decomp of V} Let $\mathfrak n$ be a maximal ideal of $\T^S$ and $V=H^i(X_U,M_{\mu,\overline\Q_p})_\n$. Then 
        \begin{equation*}
            V = \oplus_{\alpha: \T^S_{\overline\Q_p}(V) \to \overline\Q_p} V[(\ker\alpha)^\infty]
        \end{equation*}
        where $\alpha$ runs over $\overline\Q_p$-algebra homomorphisms $\T^S_{\overline\Q_p}(V)\to\overline\Q_p$ and $[(\ker\alpha)^\infty]$ is the set of elements annihilated by some power of $\ker\alpha$. There are only finitely many such $\alpha$ and any such $\alpha$ satisfies $(\alpha|_{\T^S})^{-1}(\m_{\overline\Q_p})= \n$.
        \item\label{decomp of W} Let $\mathfrak n$ be a maximal ideal of $\T^S$ and $W=H^i_!(X_U,M_{\mu,\overline\Q_p})_\n$. Then 
        \begin{equation*}
            W = \oplus_{\alpha: \T^S_{\overline\Q_p}(W) \to \overline\Q_p} W[\ker\alpha]
        \end{equation*}
        where $\alpha$ runs over $\overline\Q_p$-algebra homomorphisms $\T^S_{\overline\Q_p}(W)\to\overline\Q_p$. There are only finitely many such $\alpha$ and any such $\alpha$ satisfies $(\alpha|_{\T^S})^{-1}(\m_{\overline\Q_p})= \n$.
        \item \label{item: isotypic} Let $H=H^i_!(X_U,M_{\mu,\O})_\m[\epsilon]/\O\text{-torsion}$. Then $H[\ker\lambda]$ 
        is a free $\O$-module of rank one whose base change to $\C$ is canonically isomorphic to
        \[H^i_!(X_U,M_{\mu,\C})[(\pi^{\infty,S})^{U^S}\times \epsilon]\]
        as a $\C$-vector space.
        \item \label{eta contains eta}$\eta_{\pi,i,\epsilon}\supset \eta_\pi.$ Equality holds if $H$ is a free $T$-module of rank $1$.
    \end{enumerate}
\end{lemma}

\begin{proof}
    Note that $\T^S_{\overline\Q_p}(V)$ is finite dimensional as a $\overline\Q_p$-vector space, so it is an Artinian ring. In particular, it is the finite product of its localisations at maximal ideas. Hence we get a corresponding descomposition
    \[V=\oplus_{\p\in \mathrm{mSpec}\T^S_{\overline\Q_p}(V)} V_\p.\]
    By \cite[Section 2.5.1]{eigenbook}, $V_\p=V[\p^\infty]$.
    By Zariski's lemma, $\mathrm{mSpec}\T^S_{\overline\Q_p}(V) = \{\ker \alpha |\; \alpha:\T^S_{\overline\Q_p}(V)\to\overline\Q_p \text{ is a $\overline\Q_p$-algebra homomorphism}\}.$
    We thus get the desired decomposition of $V$.

    To see that $(\alpha|_{\T^S})^{-1}(\m_{\overline\Q_p})= \n$, note that $\T^S_{\overline\Q_p}(V) = \T^S(H^i(X_U,M_{\mu})_\n) \otimes_{\O} \overline\Q_p$ by $\O$-flatness of $\overline\Q_p$.
    Thus, $\alpha$ gives rise to an $\O$-algebra homomorphism $\T^S(H^i(X_U,M_{\mu})_\n) \to \overline\Q_p$.
    The image of the this homomorphism necessarily lies in $\O_{\overline\Q_p}$ because $\T^S(H^i(X_U,M_{\mu})_\n)$ is finitely generated as an $\O$-module.
    By the same reasoning as \Cref{tem}, $\T^S(H^i(X_U,M_{\mu})_\n)=\T^S(H^i(X_U,M_{\mu}))_\n.$
    Thus, we get an $\O$-algebra homomorphism $a:\T^S(H^i(X_U,M_{\mu}))_\n \to \O_{\overline\Q_p}$.
    We have $a^{-1}(\m_{\overline\Q_p})= \n \T^S(H^i(X_U,M_{\mu}))_\n$ by the going up theorem \cite[Lemma 00GU]{stacks-project}.
    This proves $(\alpha|_{\T^S})^{-1}(\m_{\overline\Q_p})= \n$.


    Part \ref{decomp of W} can be proved in the same way using \Cref{lem:T acts semisimply}\ref{semisimple}. One can also prove it in a more straightforward manner.

    For part \ref{item: isotypic}, we know $T[\ker\lambda]$ is a finite free $\O$-module since it is torsion-free and finitely generated.
    To find its rank, note that $\ker\lambda$ is the image of $\ker\Lambda=(t-\Lambda(t):t\in \mathcal \T^S)$ under the projection $q:\T^S\twoheadrightarrow T$.
    As $T$ is Noetherian, we can pick $t_1,\dots,t_n\in \T^S$ such that $q((t_i-\Lambda(t_i):1\le i\le n))=\ker\lambda$.
    Then 
    \[H[\ker\lambda]=H[\{t_i-\Lambda(t_i):1\le i\le n\}]=\ker(H\xrightarrow{m\mapsto (t_i-\Lambda(t_i)m)} H^n).\]
    Taking kernel commutes with flat base change, so
    \begin{align}
        H[\ker\lambda]\otimes_\O\C &= (H^i_!(X_U,M_{\mu,\O})_\m[\epsilon]\otimes_\O\C)[\ker\lambda]\nonumber\\
        &= H^i_!(X_U,M_{\mu,\C})_\m[\epsilon][\ker\lambda] \nonumber\\
        &=H^i_!(X_U,M_{\mu,\C})[\epsilon][\ker\lambda]\label{eq:localisation}\\
        &=H^i_!(X_U,M_{\mu,\C})[(\pi^{\infty,S})^{U^S}][\epsilon]\nonumber\\
        &=H^i(\mathfrak g,K_\infty^\circ,\pi_\infty\otimes_\C M_{\mu,\C})[\epsilon]\otimes_\C (\pi^\infty)^U.\label{eq:H(g,K')}
    \end{align}
    To see that equation \eqref{eq:localisation} holds, note that $\O$ is a Henselian local ring, so $\T^S(H^i_!(X_U,M_{\mu,\O})) = \oplus_\n \T^S(H^i_!(X_U,M_{\mu,\O}))_\n$, where $\n$ runs over the maximal ideal of $\T^S$ that contains $\ker(\T^S \twoheadrightarrow \T^S(H^i_!(X_U,M_{\mu,\O})))$. Thus, $H^i_!(X_U,M_{\mu,\C}) = \oplus_{\n} H^i_!(X_U,M_{\mu,\C})_\n$. By part \ref{decomp of W}, if $\n\neq \m$, then $H^i_!(X_U,M_{\mu,\C})_\n[\ker\lambda]=0$.
    Equation \eqref{eq:H(g,K')} holds by lemma \ref{lem:T acts semisimply}.
    By \Cref{1 dim space} and the choice of $U$, \eqref{eq:H(g,K')} is a one dimensional $\C$-vector space.
    This proves part \ref{item: isotypic}.

    The first part of \ref{eta contains eta} now follows from lemma \ref{lem: M lambda rk 1}. The remaining part is easy.
\end{proof}

\subsection{Betti-Whittaker periods}
We shall mostly follow \cite{Raghuram-Balasubramanyam} to define the Betti-Whittaker period for this subsection.
Fix all Haar measures as in that paper.
Let $F,\pi,\dots$ be defined as before.

\begin{definition}\label{defn:5 isoms}
    Let $i\in\{b_n,t_n\}$.
    Pick the generator $w_\infty$ (depends on $i$) of the 1-dimensional space (\Cref{1 dim space})
    \[H^{i}(\mathfrak g, K_\infty^0,W(\pi_\infty)\otimes M_{\mu,\C})[\epsilon]\] as in \cite[Section 3.2.2]{Raghuram-Balasubramanyam}.
    Fix a continuous unitary homomorphism $\psi:F\backslash \mathbb A_F\to \C^\times$ such that $\psi_v$ is non-trivial\footnote{Such a $\psi$ exists, e.g. take $\psi_v(x)=e^{-2\pi ix}$ for real $v$, $\psi_v(x)=e^{-2\pi i(x+\bar x)}$ for complex $v$, and $\psi_v(x)=e^{2\pi i Tr_{F_v/\Q_p}(x)}$ for all $v\mid p$ and all rational primes $p$, and $\psi=\prod_v \psi_v$.}for all $v$.
    We shall abuse notation and let $\psi$ also denote the corresponding standard character on the unipotent radical of $G$, i.e. $\psi(u)=\psi(u_{1,2}+u_{2,3}+\dots+ u_{n-1,n})$.
    Let $W(\pi^\infty)$ be the Whittaker model of $\pi^\infty$ with respect to $\psi^\infty$.
    Let $V_\pi$ be the subspace\footnote{Well-defined by the multiplicity one theorem.} of $L^2_0(G(F)\backslash G(\A_F),\chi)$ realizing $\pi$.
    Define $\mathcal F_{\pi^\infty,\epsilon,w_\infty,i}$ as the composition of the isomorphisms
    \begin{align*}\label{5 isoms}
        W(\pi^\infty)^U & \xrightarrow{\sim} W(\pi^\infty)^U\otimes_\C  H^{i}(\mathfrak g, K_\infty^0,W(\pi_\infty)\otimes M_{\mu,\C})[\epsilon] \nonumber\\
        & \xrightarrow{\sim} H^{i}(\mathfrak g, K_\infty^0,W(\pi)^U\otimes M_{\mu,\C})[\epsilon] \nonumber\\
        & \xrightarrow{\sim} H^{i}(\mathfrak g, K_\infty^0,V_\pi^U\otimes M_{\mu,\C})[\epsilon]\\
        &= H^{i}_{cusp}(X_U,M_{\mu,\C})[(\pi^\infty)^U\times \epsilon]\\
        & \xrightarrow{\sim} H^i_!(X_U,M_{\mu,\C})[(\pi^{\infty,S})^{U^S}\times \epsilon].
    \end{align*}
    The first map is $w^\infty\mapsto w^\infty\otimes w_\infty$; the second map is trivial; the third map is the inverse of the map
    \begin{align*}
        V_\pi &\xrightarrow{\sim} W(\pi)\\
        f &\mapsto \left(g\mapsto \int_{U_n(F)\backslash U_n(\mathbb A_F)} f(ug)\psi^{-1}(u)du\right)
    \end{align*}
    and $U_n$ is the unipotent radical of the standard Borel subgroup $B_n$.
    The last isomorphism is by \Cref{lem:T acts semisimply}.
\end{definition}

\begin{definition}\label{defn: period}
    For each finite place $v$, let $w_v$ be the essential vector\footnote{If $\pi_v$ is unramified, then $w_v$ is the unique element in $W(\pi_v,\psi_v)^{\GL_n(\O_{F_v})}$ with $w_v(I_n)=1$.} of $\pi_v$ in its $\psi_v$-Whittaker model.
    Let $w^\infty=\otimes_{v<\infty} w_v$.
    Then $w^\infty\in W(\pi^\infty)^U$ by our choice of $U$ and the definition of essential vector.
    Let $H=H^i_!(X_U,M_{\mu,\O})_\m[\epsilon]/\O\text{-torsion}$.
    By \Cref{lem: M ker lambda}, $H[\ker\lambda]\otimes_\O \C= H^i_!(X_U,M_{\mu,\C})[(\pi^{\infty,S})^{U^S}\times \epsilon].$
    We define the \emph{period} 
    \[\p_{\pi,i,\epsilon}\] 
    to be the number in $\C^\times$ such that 
    $\mathcal F_{\pi^\infty,\epsilon,w_\infty,i}(w^\infty)/\p_{\pi,i,\epsilon}$ is an $\O$-generator of $H[\ker\lambda].$
    This is well defined up to multiplication by $\O^\times$ by \Cref{lem: M ker lambda}\ref{item: isotypic}.
\end{definition}

\subsection{\texorpdfstring{$L(1,\pi,\Ad^\circ)$}{L(1,pi,Ad 0)} and congruence ideals}
Recall that $U_v=\GL_n(\O_v)$ for all finite $v\notin S$.
Let \begin{align*}
    \theta: \H^S(G^S,U^S) &\to \H^S(G^S,U^S)\\
    [U^S g U^S] &\mapsto [U^S g^{-1} U^S].
\end{align*}
This is an $\O$-algebra isomorphism.\footnote{This is a homomorphism because $\H^S(G^S,U^S)$ is commutative.}

\begin{lemma}\label{lem: eigensystem for contragredient}
    The Hecke eigensystem $\tilde \Lambda:\T^S \to \O$ attached to the contragredient $\tilde \pi$ is given by $\Lambda\circ\theta$.
\end{lemma}

\begin{proof}
    It is enough to show the analogous fact after base changing from $\O$ to $\C$ and working at a single place.
    Let $v\notin S$ be a finite place, $G=\GL_n(F_v)$, $K=\GL_n(\O_{F_v})$, $V=\pi_v$.

    It is well known that the restriction map induces an isomorphism 
    \begin{equation}\label{eq:res}
        (\widetilde{V})^K \xrightarrow{\sim} \widetilde{V^K}
    \end{equation}
    so in particular $(\widetilde{V})^K$ is 1-dimensional.

    Let $g\in G$. By the same proof as \cite[Lemma 5.5.1 (c)]{Diamond-Shurman}, there exist $g_1,\dots,g_m\in G$ such that $K g K= \sqcup_i g_i K = \sqcup K g_i$.
    For all $f\in (\widetilde{V})^K$ and $v\in V^K$,
    \begin{align*}
        ([KgK]f)(v)=(\sum g_i\cdot f)(v)=f(\sum g_i^{-1}v) = f([Kg^{-1}K]v)=\lambda([Kg^{-1}K])f(v).
    \end{align*}
    By \eqref{eq:res}, $[KgK]f=\lambda([Kg^{-1}K])f$, as desired.
\end{proof}

\begin{proposition}[Poincar\'e duality]\label{prop: Poincare duality}
    Let $d=dim X_U$.
    The cup product induces a perfect pairing
    \[[\,,\,]:H^i_c(X_U,M_\mu)/(\O-\text{tors}) \times H^{d-i}(X_U,M_\mu^\vee)/(\O-\text{tors}) \to \O\]
    where\footnote{In general, $M_\mu^\vee$ and $M_{\mu^\vee}$ are \emph{not} isomorphic.} $M_\mu^\vee=\Hom_\O(M_\mu,\O).$
    If $S$ is a finite set of finite places such that $U_v=\GL_n(\O_v)$ for all finite $v\notin S$, then 
    \[[tx,y] = [x, \theta(t)y]\]
    for all $t\in \H^S(G^S,U^S)$, $x\in H^i_c(X_U,M_\mu)/(\O-\text{tors})$, $y\in H^{d-i}(X_U,M_\mu^\vee)/(\O-\text{tors})$.
\end{proposition}

A version of this is proved in \cite[Theorem 4.8.9]{Harder-AG}, but the proof is not that easy. We shall deduce this from Verdier duality instead.

\begin{proof}
    As stated in \cite[Proposition 2.2.20]{10_authors}, we have by Verdier duality an isomorphism
    \[RHom_\O(R\Gamma_c(X_U,M_\mu),\O) \cong R\Gamma(X_U,M_\mu^\vee)[d]\]
    in the derived category of $\O$-modules $D(\O)$.
    It follows from one of the spectral sequences for $Ext$ that we have a spectral sequence
    \[E_2^{i,j}=Ext_\O^i(H^{-j}_c(X_U,M_\mu),\O)\Rightarrow E^{i+j}=H^{i+j+d}(X_U,M_\mu^\vee).\]
    Since $\O$ is a PID, the only non-zeros terms lie in $\{(i,j):0\le i \le 1, -d\le j\le 0\}$.
    For all $j\in\Z$, we have an exact sequence
    \[0\to E_2^{1,-j-1} \to E^{-j} \to E_2^{0,-j} \to 0\]
    i.e.
    \[0 \to Ext_\O^1(H^{j+1}_c(X_U,M_\mu),\O)
    \to H^{d-j}(X_U,M_\mu^\vee)
    \xrightarrow{f} \Hom_\O(H^{j}_c(X_U,M_\mu),\O)
    \to 0.\]
    Since $H^{j+1}_c(X_U,M_\mu)$ is finitely generated over $\O$ by the existence of Borel-Serre compactification, the second term is $\O$-torsion.
    On the other hand, the 4th term is $\O$-torsion free. 
    It follows that $\ker f$ is precisely the $\O$-torsion of $H^{d-j}(X_U,M_\mu^\vee)$, so $f$ induces an isomorphism
    \[\tilde f:H^{d-j}(X_U,M_\mu^\vee)/(\O-\text{tors})
    \xrightarrow{\sim} \Hom_\O(H^{j}_c(X_U,M_\mu)/(\O-\text{tors}),\O).\]
    We have a pairing 
    \begin{align*}
        H^{d-j}(X_U,M_\mu^\vee)/(\O-\text{tors}) &\times H^{j}_c(X_U,M_\mu)/(\O-\text{tors}) \to \O\\
        (a&,b)\mapsto \tilde f(a)(b).
    \end{align*}
    This is a perfect pairing because both
        $H^{d-j}(X_U,M_\mu^\vee)/(\O-\text{tors})$ and $H^{j}_c(X_U,M_\mu)/(\O-\text{tors})$ are finite free $\O$-modules and $\tilde f$ is an isomorphism.
    It is well-known that this is given by the cup product. 
    The last assertion about the action of the Hecke algebra follows from \cite[Proposition 2.2.20]{10_authors}.
\end{proof}

The cup product induces a pairing
\[[\,,\,]:H^i_!(X_U,M_\mu)/(\O-\text{tors}) \times H^{d-i}_!(X_U,M_\mu^\vee)/(\O-\text{tors}) \to \O.\]
To see this, it is enough to show that the cup product induces a pairing
\[H^i_!(X_U,M_\mu) \times H^{d-i}_!(X_U,M_\mu^\vee) \to \O.\]
We thus need to show that the map
\begin{align*}
    H^i_!(X_U,M_\mu) \times H^{d-i}_!(X_U,M_\mu^\vee) &\to H^d_c(X_U,M_\mu\otimes M_\mu^\vee)\\
    (x,y) &\mapsto x_c\cup y_c
\end{align*}
is well-defined, where $x_c$ is a lift of $x$ to $H^i_c(X_U,M_\mu)$, $y_c$ is a lift of $y$ to $H^{d-i}_c(X_U,M_\mu^\vee)$, and $\cup$ is the cup product on $H^i_c(X_U,M_\mu) \times H^{d-i}_c(X_U,M_\mu^\vee)$.
If $x'_c$ (resp. $y'_c$) is another lift of $x$ (resp. $y$) to $H^i_c(X_U,M_\mu)$ (resp. $H^{d-i}_c(X_U,M_\mu^\vee)$), then $x'_c= x_c+\epsilon$ and $y'_c=y_c+\delta$, where $\epsilon \in \ker (H^i_c(X_U,M_\mu) \to H^i(X_U,M_\mu))$ and $\delta \in \ker(H^{d-i}_c(X_U,M_\mu^\vee) \to H^{d-i}(X_U,M_\mu^\vee))$.
Recall that to form the cup product of two elements in the compactly supported cohomology, we can first map one of the elements to the usual cohomoology and then take the cup product between usual and compactly supported cohomology.
It follows that 
\[x'_c\cup y'_c =x_c\cup y_c + \epsilon\cup y_c + x_c\cup\delta + \epsilon\cup\delta =x_c\cup y_c.\]

For convenience, let 
$$\overline H^i_!(X_U,M_\mu):= H^i_!(X_U,M_\mu)/(\O-\text{tors})$$ and $\overline H^{d-i}_!(X_U,M_\mu^\vee):=H^{d-i}(X_U,M_\mu^\vee)/(\O-\text{tors})$.

Let $q_1:\T^S \twoheadrightarrow \T^S(\overline H^b_!(X_U,M_\mu))$ be the quotient map.
Note that $\T^S(\overline H^b_!(X_U,M_\mu))$ is a finite $\O$-algebra, so it is the (finite) product of its localisations at maximal ideals.
Thus,
\begin{equation*}
    \overline H^b_!(X_U,M_\mu) = \bigoplus_{\substack{\m_1\in\mathrm{mSpec} \T^S \\\m_1\supset \ker q_1}}\overline H^b_!(X_U,M_\mu)_{\m_1}.
\end{equation*}
Note also that $supp_{\T^S}(\overline H^b_!(X_U,M_\mu))= V(Ann_{\T^S}(\overline H^b_!(X_U,M_\mu))) = V(\ker q_1)$, so $\overline H^b_!(X_U,M_\mu)_{\m_1}=0$ if $\m_1\not\supset \ker q_1$.
It follows that 
\begin{equation}\label{eq: decomp}
    \overline H^b_!(X_U,M_\mu) = \bigoplus_{\m_1\in \mathrm{mSpec} \T^S}\overline H^b_!(X_U,M_\mu)_{\m_1}.
\end{equation}
We have a similar decomposition for $\overline H^{t}_!(X_U,M_\mu^\vee)$.
Thus, we can restrict $[\,,\,]$ to these direct summands to get pairings 
\begin{equation*}
    \overline H^b_!(X_U,M_\mu)_{\m_1}\times \overline H^{t}_!(X_U,M_\mu^\vee)_{\m_2} \to \O
\end{equation*}
for any $\m_1,\m_2\in\mathrm{mSpec}\T^S$.
Similarly, we have decompositions for $\overline H^b(X_U,M_\mu)$ and $\overline H^{t}_!(X_U,M_\mu^\vee)$ into direct sums of their localisations at maximal ideals of $\T^S$.
Thus, we can restrict $[\,,\,]$ to these direct summands to get pairings 
\begin{equation*}
    \overline H^b(X_U,M_\mu)_{\m_1}\times \overline H^{t}_c(X_U,M_\mu^\vee)_{\m_2} \to \O
\end{equation*}
for any $\m_1,\m_2\in\mathrm{mSpec}\T^S$.

Let $\partial X_U$ denote the boundary of the Borel-Serre compactification of $X_U$.
Let $\tilde \m=\theta(\m)\subset \T^S$, which equals $\ker(\widetilde\Lambda \mod \varpi)$ by \Cref{lem: eigensystem for contragredient}, where $\tilde \Lambda:\T^S \to \O$ is the Hecke eigensystem attached to the contragredient $\tilde \pi$.
Let $\tilde\epsilon: K_\infty/K_\infty^\circ\to \{\pm 1\}$ be the character such that for every real place $v$, if $x_v\in K_v/K_v^\circ$ is non-trivial, then $\tilde\epsilon(x_v)=(-1)^{n-1}\epsilon(x_v)$.

\begin{corollary}\label{coro: pp}
    Assume $H^b(\partial X_U,M_\mu)_\m$ is $\O$-torsion free and $p>2$.
    \begin{enumerate}[(a)]
        \item \label{item: pp}Then 
        \begin{equation}\label{Hm pp}
            [\,,\,]:\overline H^b_!(X_U,M_\mu)_\m \times \overline H^{t}_!(X_U,M_\mu^\vee)_{\tilde\m} \to \O
        \end{equation}
        and 
        \[[\,,\,]:\overline H^b_!(X_U,M_\mu)_\m[\epsilon] \times \overline H^{t}_!(X_U,M_\mu^\vee)_{\tilde\m}[\tilde \epsilon] \to \O\]
        are both perfect pairings.
        \item \label{item: theta isom}$\theta$ induces\footnote{By `induces', we mean the map given by lifting an element of $\T^S(\overline H^b_!(X_U,M_\mu)_\m)$ to $\T^S$, applying $\theta$ to it, and then projecting it to $\T^S(\overline H^t_!(X_U,M_\mu^\vee)_{\tilde\m}).$} an isomorphism $\T^S(\overline H^b_!(X_U,M_\mu)_\m)\cong \T^S(\overline H^t_!(X_U,M_\mu^\vee)_{\tilde\m}).$
    \end{enumerate}
\end{corollary}

\begin{proof}
    \begin{claim}
        Let $m_1,m_2 \subset \T^S$ be maximal ideals of $\T^S$.
        Then $$[\overline H^b(X_U,M_\mu)_{\m_1} , \overline H^{t}_c(X_U,M_\mu^\vee)_{\m_2}]=0$$ unless $\m_2=\theta(\m_1)$.
        Similarly, $$[\overline H^b_!(X_U,M_\mu)_{\m_1} , \overline H^{t}_!(X_U,M_\mu^\vee)_{\m_2}]=0$$ unless $\m_2=\theta(\m_1)$.
    \end{claim}
    We shall show the first one. The second one is similar and easier.
    To show this, we let $A:= H^b(X_U,M_{\mu,\overline\Q_p})_{\m_1}, B:= H^{t}_c(X_U,M_{\mu,\overline\Q_p}^\vee)_{\m_2}$.
    Suppose $[A,B]\neq 0$.
    By \Cref{lem: M ker lambda}\ref{decomp of V} and its analogue for the compactly supported cohomology, there exist $a\in A[(\ker \alpha)^\infty], b\in B[(\ker \beta)^\infty]$ such that $[a,b]\neq 0$, where $\alpha,\beta:\T^S_{\overline\Q_p} \to \overline\Q_p$ are $\overline\Q_p$-algebra homomorphisms with $(\alpha|_{\T^S})^{-1}(\m_{\overline\Q_p})= \m_1$ and $(\beta|_{\T^S})^{-1}(\m_{\overline\Q_p})= \m_2$.
    Let $t\in \T^S$. Then $t-\alpha(t)\in\ker\alpha$, so $(t-\alpha(t))^r a=0$ for some $r\ge 1$.
    For all $y\in B[(\ker \beta)^\infty]$, we have
    \[ 0  = [(t-\alpha(t))^r a,y] = [x, (\theta(t)-\alpha(t))^r y]\]
    Thus, $(\theta(t)-\alpha(t))^r$ is not surjective as an $\overline\Q_p$-linear endomorphism of $B[(\ker \beta)^\infty]$, so it is not injective, so $(\theta(t)-\alpha(t))^r c=0$ for some $c\in B[(\ker \beta)^\infty]\setminus\{0\}$.
    By the definition of $B$, there exists $s\ge 1$ such that $(\theta(t)-\beta(\theta(t)))^s c=0$.
    Since $c\neq 0$, we must have $\alpha(t) = \beta(\theta(t))$.
    Thus, $\alpha|_{\T^S}=\beta|_{\T^S}\circ\theta$.
    It follows that $\m_1 = (\alpha|_{\T^S})^{-1}(\m_{\overline\Q_p}) = (\beta|_{\T^S}\circ\theta)^{-1}(\m_{\overline\Q_p}) = \theta(\m_2)$.

    Using this claim, we shall now deduce the first part of \ref{item: pp} by a similar argument to \cite[section 4.2.4]{Raghuram-Balasubramanyam} under our assumption that $H^b(\partial X_U,M_\mu)_\m$ is $\O$-torsion free.

    We first note that the base change of \eqref{Hm pp} to $E$
    \begin{equation}\label{E perf pair}
        [\,,\,]:H^b_!(X_U,M_{\mu,E})_\m \times H^{t}_!(X_U,M_{\mu,E}^\vee)_{\tilde\m} \to E
    \end{equation}
    is a perfect pairing, because this is non-degenerate by Poincar\'e duality and the claim proved above.
    Similarly, 
    \begin{equation}\label{E perf pair2}
        [\,,\,]:H^b(X_U,M_{\mu,E})_\m \times H^{t}_c(X_U,M_{\mu,E}^\vee)_{\tilde\m} \to E
    \end{equation}
    is a perfect pairing.

    Next, note that since $\overline H^b_!(X_U,M_\mu)_\m$ and $\overline H^{t}_!(X_U,M_\mu^\vee)_{\tilde\m}$ are both finite free $\O$-modules, to show that \eqref{Hm pp} is a perfect pairing, it is enough to show that the induced map
    \begin{equation}\label{eq: perf pair}
        \overline H^b_!(X_U,M_\mu)_\m \to \Hom_{\O}(\overline H^{t}_!(X_U,M_\mu^\vee)_{\tilde\m} , \O)
    \end{equation}
    is an isomorphism.
    To see injectivity, note that $\overline H^b_!(X_U,M_\mu)_\m \hookrightarrow \overline H^b_!(X_U,M_\mu)_\m \otimes_{\O}E = H^b_!(X_U,M_{\mu,E})_\m$ and similarly for the right hand side of \eqref{eq: perf pair}. Thus, we deduce the injectivity of \eqref{eq: perf pair} from the non-degeneracy of \eqref{E perf pair}.

    To check surjectivity of \eqref{eq: perf pair}, let $f\in \Hom_{\O}(\overline H^{t}_!(X_U,M_\mu^\vee)_{\tilde\m} , \O)$.
    Analogous to \eqref{eq: decomp}, $\overline H^{t}_!(X_U,M_\mu^\vee)$ is the direct sum of its localisations at maximal ideals of $\T^S$.
    We extend $f$ to $f': \overline H^{t}_!(X_U,M_\mu^\vee) \to \O$ by setting it to be $0$ on other summands.
    Let $q: \overline H^{t}_c(X_U,M_\mu^\vee) \twoheadrightarrow \overline H^{t}_!(X_U,M_\mu^\vee)$ be the natural map.
    By \Cref{prop: Poincare duality}, there exists $x\in \overline H^b(X_U,M_\mu)$ such that $$f'\circ q=[x,-].$$
    By the claim, we can assume that $x$ lies in the direct summand $\overline H^b(X_U,M_\mu)_\m$ of $\overline H^b(X_U,M_\mu)$.
    On the other hand, by the perfectness of \eqref{eq: perf pair}, there exists $z\in H^b_!(X_U,M_{\mu,E})_\m$ such that $f_E = [z,-]$, where $f_E$ is the base change of $f$ along $\O\to E$.
    Then
    \[f'_E\circ q_E = [z,-]\]
    as maps $H^{t}_c(X_U,M_{\mu,E}^\vee) \to E$.
    By the perfectness of \eqref{E perf pair2}, $x_E = z$,
    where $x_E$ is the image of $x$ in $H^b(X_U,M_{\mu,E})_\m$.
    In particular, $$x_E \in H^b_!(X_U,M_{\mu,E})_\m.$$
    Recall that there is a long exact sequence $$\dots\to H^i_c(X_U,M_\mu)\to H^i(X_U,M_\mu) \to H^i(\partial X_U,M_\mu)\to \dots$$
    We lift $x$ to $H^b(X_U,M_\mu)_{\m}$ and consider its image $y$ in $H^b(X_U,M_\mu)_{\m}/H^b_!(X_U,M_\mu)_{\m} \subset H^b(\partial X_U,M_\mu)_{\m}$.
    By what we have just shown, $y_E \in H^b(\partial X_U,M_{\mu,E})_{\m}$ is $0$.
    Since $H^b(\partial X_U,M_\mu)_{\m}$ is $\O$-torsion-free, $y=0$.
    Thus, $x \in \overline H^b_!(X_U,M_\mu)_\m$.

    For the second part of \ref{item: pp}, note that we have a decomposition
    \[\overline H^b_!(X_U,M_\mu)_\m = \bigoplus_{\epsilon_1\in \widehat{K_\infty/K_\infty^\circ}}\overline H^b_!(X_U,M_\mu)_{\m}[\epsilon_1]\]
    since $p>2$.
    The perfectness then follows from the first part and the proof of \cite[Proposition 3.3.1]{Raghuram-Balasubramanyam}. (The proof there works here in view of the decomposition of $H^{i}(\mathfrak g, K_\infty^0,\pi_\infty\otimes_\C M_{\mu,\C})$ in the last part of the proof of \ref{1 dim space}.)

    Part \ref{item: theta isom} follows from part \ref{item: pp} and the same argument as \cite[Corollary 2.2.21]{10_authors}, namely the commutativity of the diagram
\[\begin{tikzcd}
	{\mathbb T^S} & {End_{\mathcal O}(\overline H^b_!(X_U,M_\mu)_{\mathfrak m})} \\
	{\mathbb T^S} & {End_{\mathcal O}(\overline H^t_!(X_U,M_\mu^\vee)_{\tilde{\mathfrak m}})}
	\arrow[from=1-1, to=1-2]
	\arrow["\theta"', from=1-1, to=2-1]
	\arrow["{\text{transpose}}", from=1-2, to=2-2]
	\arrow[from=2-1, to=2-2]
\end{tikzcd}\]
\end{proof}

\begin{lemma}\label{lem: petersson}
    Let $\phi,\tilde\phi$ be the specific cusp forms in the space of cusp forms affording $\pi,\tilde\pi$ respectively which are defined in \cite[Section 2.2]{Raghuram-Balasubramanyam}.\footnote{The corresponding Whittaker vectors of $\phi,\tilde\phi$ are tensor products of local Whittaker vectors. At finite places, the local Whittaker vectors are the essential vectors. At infinite places, the local Whittaker vectors are the `cohomological vectors' as defined in \cite{Raghuram-Balasubramanyam}.}
    Define
    \[\langle \phi,\tilde\phi \rangle = \int_{A_G G(F)\backslash G(\A_F)} \phi(g)\tilde\phi(g) dg.\]
    Then 
    \[\langle \phi,\tilde\phi \rangle = \frac{\prod_{v\mid \infty}\mathfrak c_v^\sharp(w_v,\tilde w_v) \cdot L(1,\pi,\Ad^\circ)}{\alpha_F\; \p_{ram}(\pi)},\]
    where $\alpha_F:=\frac{\hat\Phi_f(0)}{n \Res_{s=1}\tilde\zeta_F(s)}$ with $\Phi_f$ the characteristic function of $\prod_{v\nmid\infty}\O_{F_v}^n$, $\hat\Phi_f$ its Fourier transform, and $\tilde\zeta_F$ the completed zeta function. The measures, $\mathfrak c_v^\sharp(w_v,\tilde w_v)$, and $\p_{ram}(\pi)$ are defined as in \cite[section 2]{Raghuram-Balasubramanyam}. Also, $L(1,\pi,\Ad^\circ)$ is the value at $1$ of the Langlands $L$-function\footnote{The Langlands $L$-function is defined as the product over all places of $F$ of the local $L$-factors. The local $L$-factor can be defined as the local $L$-factor of the corresponding Weil(-Deligne) representation under the local Langlands correspondence at a (non-)Archimedean place.}.
\end{lemma}

\begin{proof}
    Note that $\int_{A_G G(F)\backslash G(\A_F)} = \int_{Z(\A_F) G(F)\backslash G(\A_F)} \int_{A_G G(F)\backslash Z(\A_F)G(F)}$, where $Z$ is the centre of $G$.
    Also, $A_G G(F)\backslash Z(\A_F)G(F) = F^\times \backslash \A^1_F$, where $\A^1_F:=\{x\in \A_F: |x|_{\A_F}=1\}.$
    It follows that  \[\langle \phi,\tilde\phi \rangle = vol(F^\times\backslash \A_F^1)\int_{Z(\A_F)G(F)\backslash G(\A_F)} \phi(g)\tilde\phi(g) dg.\]
    The result now follows from \cite[equation (2.2.11)]{Raghuram-Balasubramanyam}.
    They obtained their result by relating the Petersson inner product with $L(1,\pi,\Ad^0)$ by the Rankin-Selberg method and using the fact that $L(s,\pi\times \tilde\pi)=\tilde\zeta_F(s)L(s,\pi,\Ad^0).$
\end{proof}

\begin{lemma}\label{lem: pairing-L}
    We have
    \[[\vartheta^\circ_{b,\epsilon},\tilde\vartheta^\circ_{t,\tilde\epsilon}] = L^{alg}(1,\pi,\ad^0,\epsilon),\]
    where $[\;,\;]$ is the pairing induced by cup product as before, $\vartheta^\circ_{b,\epsilon}$ is an $\O$-basis of $\overline H^b_!(X_U,M_\mu)_\m[\epsilon][\ker\lambda]$, $\tilde\vartheta^\circ_{t,\tilde\epsilon}$ is an $\O$-basis of $\overline H^t_!(X_U,M_\mu^\vee)_{\tilde\m}[\tilde\epsilon][\ker\tilde\lambda]$, and \[L^{alg}(1,\pi,\Ad^0,\epsilon):=\frac{L(1,\pi,\Ad^0)}{\alpha_F \,\p_{ram}(\pi)\p_\infty(\pi)\p_{\pi,b,\epsilon}\p_{\tilde\pi,t,\tilde\epsilon}}.\]
    Here, $\p_\infty(\pi)$ is defined as in \cite[equation (3.3.9)]{Raghuram-Balasubramanyam} and $\p_{\pi,b,\epsilon}, \p_{\tilde\pi,t,\tilde\epsilon}$ are defined in \Cref{defn: period}.\footnote{However, our $L^{alg}(1,\pi,\Ad^0,\epsilon)$ and periods are slightly different from that in \cite{Raghuram-Balasubramanyam} due to our different choice of $X_U$.}
\end{lemma}

\begin{proof}
    This is more or less what \cite[section 3.3.3]{Raghuram-Balasubramanyam} obtained, except that our space $X_U$ is different from the locally symmetric spaces they used. As in \cite[p.658]{Raghuram-Balasubramanyam}, we have 
    \begin{align*}
        [\p_{\pi,b,\epsilon}\vartheta^\circ_{b,\epsilon},\p_{\tilde\pi,t,\tilde\epsilon}\tilde\vartheta^\circ_{t,\tilde\epsilon}] &= \frac{1}{vol(U)} \int_{G(F) \backslash G(\A_F)/K_\infty^\circ U} \varsigma\\
        &= \frac{1}{vol(U)}\int_{G(F) \backslash G(\A_F)/A_G U} \varsigma\\
        &= \int_{G(F) \backslash G(\A_F)/A_G} \varsigma\\
        &= \frac{L(1,\pi,\Ad^0)}{\alpha_F \,\p_{ram}(\pi)\p_\infty(\pi)}
    \end{align*}
    where $\varsigma$ has the same meaning as that in \cite[p.658]{Raghuram-Balasubramanyam} and in the last equality we used \Cref{lem: petersson} instead of \cite[equation (2.2.11)]{Raghuram-Balasubramanyam}.
    Dividing both sides by $\p_{\pi,b,\epsilon}\p_{\tilde\pi,t,\tilde\epsilon}$ gives the result.
\end{proof}

\begin{remark}
    $\alpha_F$ depends only on $F$, $\p_{ram}(\pi)$ depends only on the ramified components of $\pi$, and $\p_{\infty}(\pi)$ depends only on $\pi_\infty$.
\end{remark}

We can now prove our first main theorem.
\begin{theorem}\label{thm: main}
    Let the assumptions be as at the start of \Cref{sec:4}.
    Suppose that $H^b(\partial X_U,M_\mu)_\m$ is $\O$-torsion free.
    Then 
    \[\eta_{\pi,b,\epsilon}=\eta_{\tilde\pi,t,\tilde\epsilon}=(L^{alg}(1,\pi,\Ad^0,\epsilon)).\]
\end{theorem}

\begin{proof}
    Since $H^b(\partial X_U,M_\mu)_\m$ is $\O$-torsion free, the cup product gives a perfect pairing
    \[[\,,\,]:\overline H^b_!(X_U,M_\mu)_\m[\epsilon] \times \overline H^{t}_!(X_U,M_\mu^\vee)_{\tilde\m}[\tilde \epsilon] \to \O\]
    by \Cref{coro: pp}.\footnote{This is the only place in this proof where we use the assumption that $H^b(\partial X_U,M_\mu)_\m$ is $\O$-torsion free.}
    We would like to apply \Cref{lem: pp}.
    The conditions in part (b) of that lemma is satisfied by \Cref{lem: M ker lambda} and the fact that $(\pi^\infty)^U$ and $(\tilde\pi^\infty)^U$ are one dimensional.

    Recall that $\lambda:T\to \O$ is the Hecke eigensystem for $\pi$.
    This factors through $\T^S(\overline H^b_!(X_U,M_\mu)_\m)$ because the action of the Hecke algebra on $\overline H^*_!(X_U,M_\mu)_\m$ preserves degree and $\pi$ is isomorphic to a submodule of $\overline H^b_!(X_U,M_\mu)_\m\otimes_\O \C$.

    Similar to how we defined the idempotent $e\in T_E$ using $\lambda:T\to \O$, we can define an idempotent $e'\in \T^S(\overline H^b_!(X_U,M_\mu)_\m)_E$ using the induced map on the quotient $\T^S(\overline H^b_!(X_U,M_\mu)_\m)\to \O$.
    It is clear from the definitions of $e$ and $e'$ that $e'$ is the image of $e$ under $T_E \twoheadrightarrow \T^S(\overline H^b_!(X_U,M_\mu)_\m)_E$.

    The same argument (with $H^b_!(X_U,M_\mu)_\m$ replaced by $H^t_!(X_U,M_\mu^\vee)_{\tilde\m}$) works for the contragredient $\tilde\pi$, $\tilde\lambda$, $\tilde e, \tilde e'$ and we know $\tilde e'$ is the image of $\tilde e$ under $\tilde T_E \twoheadrightarrow \T^S(\overline H^t_!(X_U,M_\mu^\vee)_{\tilde\m})_E$.

    By \Cref{lem: eigensystem for contragredient}, we know $\tilde\lambda=\lambda\circ\theta$. 
    It follows from definition and part (b) of \Cref{coro: pp} that $\theta(e')=\tilde e'$.
    Thus, for all $x\in\overline H^b_!(X_U,M_\mu)_\m[\epsilon]$, $y \in \overline H^{t}_!(X_U,M_\mu^\vee)_{\tilde\m}[\tilde \epsilon]$, we have
    \[[ex,y]=[e'x,y]=[x,\theta(e')y]=[x,\tilde e' y]=[x,\tilde e y].\]
    From this\footnote{The reason that we need to argue via $e', \tilde e'$ rather than $e,\tilde e$ directly is that we do not have the analogue of part (b) of \Cref{coro: pp} for the entire inner cohomology. We only have it for the bottom and top degrees.} and the fact that $e,e'$ are idempotents, we deduce that all the conditions of \Cref{lem: pp} are satisfied.

    By \Cref{lem: pp} and \Cref{lem: pairing-L}, 
    \[\eta_{\pi,b,\epsilon}=\eta_{\tilde\pi,t,\tilde\epsilon}= [\vartheta^\circ_{b,\epsilon},\tilde\vartheta^\circ_{t,\tilde\epsilon}] =(L^{alg}(1,\pi,\Ad^0,\epsilon)).\]
\end{proof}

\begin{remark}
    The theorem is a generalisation of \cite[Proposition 4.12 first part and Lemma 5.6(iv)]{tu}, where analogous results were obtained for $\GL_2$ over a totally real field and over an imaginary quadratic field.\footnote{At least in the totally real case, although it was not explicitly stated, they implicitly assumed that the ideal $\m$ is non-Eisenstein, as they applied the Poincaré duality result from \cite[Proposition 4.10 part 3]{tu}, which was established only under this assumption in their paper.}

An analogous result is stated in \cite[section 4]{Raghuram-Balasubramanyam}. It differs from our result in the following ways: the result there is for the product of $L^{alg}(1,\pi,\Ad^0,\epsilon)$ over all permissible $\epsilon$ while our result is for each individual permissible $\epsilon$.
Additionally, we localize at a maximal ideal throughout, which is necessary for relating congruence ideals to Selmer groups (see \Cref{thm: selmer} below).
Localisation also makes some results slightly harder to prove, but requires a weaker hypothesis that $H^b(\partial X_U,M_\mu)_\m$, rather than $H^b(\partial X_U,M_\mu)$, is $\O$-torsion free (see \Cref{lem: boundary vanishes} below for a case where this holds).  Although this weaker hypothesis was mentioned in \cite[page 669]{Raghuram-Balasubramanyam}, the necessary (small) adjustments to their proof were not worked out.
We have provided more detailed arguments for certain parts.
Furthermore, to facilitate the use of some results in \cite{10_authors} and \cite{Newton-Thorne-derived}, we work with a different locally symmetric space, preventing us from applying some results from \cite{Raghuram-Balasubramanyam} directly.
For the same reason, our $L^{alg}(1,\pi,\Ad^0,\epsilon)$ is slightly different from theirs. 
\end{remark}

\begin{remark}
    It may be possible to determine $\p_\infty(\pi)$ explicitly as a power of $2\pi i$ using techniques of \cite{Grobner-Lin}, but we have not attempted this.
    In that paper, they precisely determined some archimedean zeta-integrals by replacing $\pi$ with simpler automorphic representations $\pi'$ with $\pi_\infty \cong \pi_\infty'$. Here, 'simpler' means $\pi'$ is automorphically induced from a Hecke character or is an isobaric sum of Hecke characters.
    This approach allows them to relate the $L$-function of $\pi$ to those of Hecke characters, which, in turn, are related to CM periods by results of Blasius.
\end{remark}

Now, we will explain why this is related to congruences of automorphic representations.
Roughly speaking, if $L^{alg}(1,\pi,\Ad^0,\epsilon)$ is not a $p$-adic unit, then $\pi$ is congruent to another automorphic representation. The converse holds if the maximal ideal $\m$ is non-Eisenstein.
\begin{corollary}\label{coro: congruence}
    Let the assumptions be as at the start of \Cref{sec:4}.
    Suppose that $H^b(\partial X_U,M_\mu)_\m$ is $\O$-torsion free.
    Consider the following two statements:
    \begin{enumerate}[(a)]
        \item $\varpi\mid L^{alg}(1,\pi,\Ad^0,\epsilon)$. \label{list 1}
        \item There is a discrete automorphic representation $\pi'\not\cong \pi$ of $\GL_n(\A_F)$ with $H^*_!(X_U,M_{\mu,\C})[(\pi'^{\infty,S})^{U^S}]\neq 0$ whose Hecke eigensystem\footnote{This is defined using the fixed isomorphism $\iota:\overline\Q_p\xrightarrow{\sim} \C$ as in \Cref{lem:Lambda}. The only differences are that $(\pi'^{\infty,S})^{U^S}$ only appears in the inner cohomology but not the cuspidal cohomology, and the image of $\Lambda'$ needs not lie in $\O$, but only an integral extension.} $\Lambda':\T^S\to \overline\Q_p$ satisfies $|\Lambda(t)-\Lambda'(t)|_p<1$ for all $t\in \T^S$. \label{list 2}
    \end{enumerate}
    Then \ref{list 1} implies \ref{list 2}. 
    If $H^b_!(X_U,M_{\mu,\O})_\m[\epsilon]/\O\text{-torsion}$ is a free $T$-module, then \ref{list 2} implies \ref{list 1}. 
\end{corollary}

Note that $\pi'$ needs not be cuspidal even though we start with a cuspidal $\pi$. See \Cref{coro: cuspidal congruence} below however.

\begin{proof}
    Abusing notation, we shall identify $\overline\Q_p$ with $\C$ using our fixed isomorphism $\iota$.
    Note that $\pi'\cong \pi$ if and only if $\Lambda'=\Lambda$ by the strong multiplicity one theorem (where we regard $\Lambda$ as having codomain in $\overline\Q_p$). 

    Suppose $\varpi\mid L^{alg}(1,\pi,\Ad^0,\epsilon)$.
    By \Cref{thm: main} and \Cref{lem: M lambda rk 1}, $\eta_\pi\neq \O$.
    By \Cref{remark: congruence} there exists a $\overline\Q_p$-algebra homomorphism $\lambda':T\otimes_\O \overline\Q_p \to \overline\Q_p$ such that $\lambda'\neq \lambda\otimes_{\O}\overline\Q_p$ and $|\lambda(t)-\lambda'(t)|<1$ for all $t\in T$.
    We have a natural map
    $$\T^S(H^*_!(X_U,M_\mu))\otimes_\O \overline\Q_p = \T^S_{\overline\Q_p}(V)\to T\otimes_\O \overline\Q_p,$$
    where $V:=H^*_!(X_U,M_{\mu,\overline\Q_p})$ and $\T^S_{\overline\Q_p}(V):= im(\T^S\otimes_\O \bar\Q_p \to \End_{\bar\Q_p}(V))$.
    Composing this with $\lambda'$, we get a $\overline\Q_p$-algebra homomorphism
    \[f:\T^S_{\overline\Q_p}(V)\to \overline\Q_p.\]
    Let $\n:=\ker f$.
    Then $$V_\n\neq 0$$ because $Supp_{\T^S_{\overline\Q_p}(V)}(V)= \{\p\in Spec(\T^S_{\overline\Q_p}(V)): \p \supset Ann_{\T^S_{\overline\Q_p}(V)}(V)\}=Spec(\T^S_{\overline\Q_p}(V)).$
    By \cite[section 2.5.1]{eigenbook}, $V[\n]\neq 0$.
    By \Cref{lem:T acts semisimply} part \ref{semisimple}, there is a discrete automorphic representation $\pi'$ such that $(\pi'^{\infty,S})^{U^S}$ is isomorphic to a sub-$\T^S_{\overline\Q_p}(V)$-module of $V$ and $(\pi'^{\infty,S})^{U^S}[\n]\neq 0.$
    Note that $(\pi'^{\infty,S})^{U^S} = \otimes'_v (\pi'_v)^{U_v}$ is one dimensional over $\overline\Q_p$, so $(\pi'^{\infty,S})^{U^S}[\n]= (\pi'^{\infty,S})^{U^S}.$
    The Hecke eigensystem attached to $\pi'$ has the desired property. 

    Suppose $H^b_!(X_U,M_{\mu,\O})_\m[\epsilon]/\O\text{-torsion}$ is a free $T$-module, say, of rank $d$, and there is a $\pi'$ satisfying the statement of the corollary.
    Note that $d\ge 1$ by \Cref{lem: M ker lambda}\ref{item: isotypic}.
    By freeness, $\eta_{\pi,b,\epsilon}=\eta_\pi^d$, so by \Cref{thm: main}, it suffices to show that $\eta_\pi\neq \O$.
    By \Cref{lem:congruence}, it suffices to show that there is an $\O$-algebra homomorphism $\lambda':T\to \overline\Q_p$ with $\lambda\neq \lambda'$ and $|\lambda(t)-\lambda'(t)|<1$ for all $t\in T$.
    Equivalently, we need to show that there is an $\O$-algebra homomorphism $\Lambda'':\T^S\to \overline\Q_p$ that factors through $\T^S(H^*_!(X_U,M_{\mu}))$ with $\Lambda\neq \Lambda''$ and $|\Lambda(t)-\Lambda''(t)|<1$ for all $t\in \T^S$, because any such $\Lambda''$ necessarily factors through $T$.
    For this, we can take $\Lambda''$ to be the Hecke eigensystem attached to $\pi'$.
\end{proof}

\begin{remark}
    A version of \Cref{coro: congruence} appears in \cite[Theorem 4.3.1]{Raghuram-Balasubramanyam}, but the conditional converse is not stated explicitly and is not proved.
An analogue of \Cref{coro: congruence} in the case of $\GL_2$ is also proved in \cite[Theorem 5.25]{namikawa_congruence_2015} under certain conditions for minimal and ordinary eigenforms.

We think the condition for the converse, namely that $H^b_!(X_U,M_{\mu,\O})_\m[\epsilon]/\O\text{-torsion}$ is a free $T$-module, is not unreasonable. When $F$ is a CM field, it should be possible to verify this using the Taylor-Wiles method if $\m$ is non-Eisenstein and the Galois representation attached to $\pi$ is a minimally ramified deformation of its residual representation.
However, under our current knowledge of Galois representations, such an approach will require a lot of extra assumptions and conjectures (such as the vanishing of $H^i(X_U,k)_\m$ outside the cuspidal range, existence of Hecke algebra valued Galois representations (without nilpotent ideals), local-global compatibility of such representations). Therefore, we do not pursue this approach here. See, however, \cite{hansen_minimal} for the $\GL_2$ case.\footnote{We think that in that (well-written) paper, it is necessary for $p\ge 7$ instead of $p\ge 3$ as stated. This ensures the image of the residual Galois representation is enormous, guaranteeing the existence of Taylor-Wiles primes. Also, it seems that for the first equation on page 8 to hold, one should patch the (derived) dual of $C^\bullet_n$ rather than $C^\bullet_n$ itself.} It may also be possible to extend the freeness to the non-minimal case using the methods of \cite{Manning1,Manning2}.
\end{remark}

Now, assume in addition that $F$ is a CM field or a totally real number field, and $S$ contains $\{\text{places }w: w|_{\Q} \text{ is ramified in $F$ or }w|_{\Q}=p\}$. (Recall that we also make the assumption that $S$ contains all $v$ such that $\pi_v$ is ramified.)
By \cite[Theorem A]{HLTT}, there is a Galois representation $$\rho_\pi:\Gal(\overline F/F) \to \GL_n(\overline \Q_p)$$ attached to $\pi$ such that for all $v\notin S$ of $F$, the characteristic polynomial of the image of $\Frob_v$ is 
\[X^n-T_{v,1}X^{n-1}+\dots+(-1)^i q_v^{i(i-1)/2}T_{v,i}X^{n-i}+\dots+ (-1)^n q_v^{n(n-1)/2}T_{v,n},\]
where $T_{v,i}=[\GL_n(\O_{F_v})diag(\varpi_v,\dots,\varpi_v,1,\dots,1)\GL_n(\O_{F_v})]$ with $\varpi_v$ appearing $i$ times.
From now on, we assume that $\m$ is \emph{non-Eisenstein}, i.e. the residual Galois representation $\overline\rho_\pi:\Gal(\overline F/F)\to \GL_n(\overline{ \mathbb F}_p)$ is (absolutely) irreducible.

\begin{lemma}\label{lem: boundary vanishes}
    Let the assumptions be as above (so in particular $\m$ is non-Eisenstein).
    Then 
    \begin{enumerate}[(a)]
        \item $H^*(\partial X_U,M_{\mu})_\m=0$,
        \item $H^*(X_U,M_{\mu})_\m=H^*_!(X_U,M_{\mu})_\m$,\label{two equal}
        \item $T=\T^S(H^*(X_U,M_{\mu}))_{\m}/\O\text{-tors}$,\label{alt T}
        \item $H^*(X_U,M_{\mu,\C})_\m=H^*_!(X_U,M_{\mu,\C})_\m=H^*_{cusp}(X_U,M_{\mu,\C})_\m$. \label{all equal}
    \end{enumerate}
\end{lemma}

\begin{proof}
    The key input for the first part is \cite[Theorem 4.2]{Newton-Thorne-derived}, which states that for every smooth $\O[U_S]$-module $A$ that is finite as an $\O$-module, $H^*(\partial X_U,A)_\m=0.$
    They proved this using the fact that $\partial X_U$ admits a stratification with strata indexed by conjugacy classes of proper parabolic subgroups of $\GL_n$ and by an in depth analysis of the cohomology of each stratum.

    Recall that by \cite[page 19]{Newton-Thorne-derived}, $M_\mu \otimes_\O \O/\varpi=M_{\mu,\O/\varpi},$ which receives an action of $\GL_n(\O/\varpi)$, compatible with that of $\GL_n(\O)$.
    Thus $M_{\mu,\O/\varpi}$ is a smooth $\O[U_S]$-module that is finite as $\O$-module, so $$H^*(\partial X_U,M_{\mu,\O/\varpi})_\m=0.$$
    We get the desired result by considering the long exact sequence associated to $$0\to M_\mu \xrightarrow{\varpi} M_\mu \to M_{\mu,\O/\varpi} \to 0$$
    and applying the Nakayama lemma to the finite $\O$-module $H^*(\partial X_U,M_{\mu})_\m=0.$

    Part \ref{two equal} now follows from the long exact sequence
    \[\dots\to H^i_c(X_U,M_\mu) \to H^i(X_U,M_\mu) \to H^i(\partial X_U,M_\mu)\to \dots\]

    For part \ref{alt T}, note that $T^S(H)_\m/\O\text{-tors}\cong \T^S(\overline{H_\m})$ for $H\in\{H^*(X_U,M_{\mu}),H^*_!(X_U,M_{\mu})\}$ by the same proof as \Cref{tem}

    For part \ref{all equal}, it is proved in \cite[Theorem 2.4.10]{10_authors} using Franke's decomposition of $H^*(X_U,M_{\mu,\C})$ via automorphic forms that $H^*(X_U,M_{\mu,\C})_\m=H^*_{cusp}(X_U,M_{\mu,\C})_\m.$
    As $H^*_{!}(X_U,M_{\mu,\C})_\m$ is always sandwiched between these two groups, these groups are all equal. (The first equality also follows from part \ref{two equal}.)
\end{proof}

For readers' convenience, we restate our running assumptions.
\begin{corollary}\label{coro: cuspidal congruence}
    Let the assumptions be as at the start of \Cref{sec:4}.
    Suppose in addition that $F$ is a CM or totally real number field, and $S$ contains $\{\text{places }w: w|_{\Q} \text{ is ramified in $F$ or }w|_{\Q}=p\}$ (and all $v$ such that $\pi_v$ is ramified).

    Assume that $\m$ is non-Eisenstein.
    Consider the following two statements:
    \begin{enumerate}[(a)]
        \item \[\varpi\mid L^{alg}(1,\pi,\Ad^0,\epsilon).\] \label{list a}
        \item There is a cohomological cuspidal automorphic representation $\pi'\not\cong \pi$ of weight $\iota\mu$ of $\GL_n(\A_F)$ with $(\pi')^U\neq 0$ whose Hecke eigensystem $\Lambda':\T^S\to \overline\Q_p$ satisfies $|\Lambda(t)-\Lambda'(t)|_p<1$ for all $t\in \T^S$. \label{list b}
    \end{enumerate}
    Then \ref{list a} implies \ref{list b}. 
    If $H^b_!(X_U,M_{\mu,\O})_\m[\epsilon]/\O\text{-torsion}$ is a free $T$-module, then \ref{list b} implies \ref{list a}. 
\end{corollary}

\begin{proof}
    The proof is just a slight variation of that of \Cref{coro: congruence}.
    Abusing notation, we shall identify $\overline\Q_p$ with $\C$ using our fixed isomorphism $\iota$.
    Suppose $\varpi\mid L^{alg}(1,\pi,\Ad^0,\epsilon)$.
    By \Cref{thm: main} and \Cref{lem: M lambda rk 1}, $\eta_\pi\neq \O$.
    By \Cref{remark: congruence} there exists a $\overline\Q_p$-algebra homomorphism $\lambda':T\otimes_\O \overline\Q_p \to \overline\Q_p$ such that $\lambda'\neq \lambda\otimes_{\O}\overline\Q_p$ and $|\lambda(t)-\lambda'(t)|<1$ for all $t\in T$.
    We have a natural map 
    $$\T^S_{\overline\Q_p}(V)\to T\otimes_\O \overline\Q_p=\T^S_{\bar\Q_p}(H^*_{cusp}(X_U,M_{\mu,\bar\Q_p})_\m),$$
    where the last equality is by \Cref{lem: boundary vanishes}, $V:=H^*_{cusp}(X_U,M_{\mu,\overline\Q_p})$, and $\T^S_{\overline\Q_p}(V):= im(\T^S\otimes_\O \bar\Q_p \to \End_{\bar\Q_p}(V))$.
    Composing this with $\lambda'$, we get a $\overline\Q_p$-algebra homomorphism
    \[f:\T^S_{\overline\Q_p}(V)\to \overline\Q_p.\]
    Let $\n:=\ker f$.
    Then $$V_\n\neq 0$$ because $Supp_{\T^S_{\overline\Q_p}(V)}(V)= \{\p\in Spec(\T^S_{\overline\Q_p}(V)): \p \supset Ann_{\T^S_{\overline\Q_p}(V)}(V)\}=Spec(\T^S_{\overline\Q_p}(V)).$
    By \cite[section 2.5.1]{eigenbook}, $V[\n]\neq 0$.
    Thus, there is a cohomological cuspidal automorphic representation $\pi'$ of weight $\iota\mu$ such that $(\pi'^{\infty,S})^{U^S}$ is isomorphic to a sub-$\T^S_{\overline\Q_p}(V)$-module of $V$ and $(\pi'^{\infty,S})^{U^S}[\n]\neq 0.$
    Note that $(\pi'^{\infty,S})^{U^S}= \otimes'_v (\pi'_v)^{U_v}$ is one dimensional over $\overline\Q_p$, so $(\pi'^{\infty,S})^{U^S}[\n]= (\pi'^{\infty,S})^{U^S}.$
    The Hecke eigensystem attached to $\pi'$ has the desired property. 

    When $H^b_!(X_U,M_{\mu,\O})_\m[\epsilon]/\O\text{-torsion}$ is a free $T$-module, the converse follows from \Cref{coro: congruence} because if $\pi'$ is a cohomological cuspidal automorphic representation $\pi'$ of weight $\iota\mu$ of $\GL_n(\A_F)$ with $(\pi')^U\neq 0$, then $\pi'$ is a discrete automorphic representation and $0\neq H^b_{cusp}(X_U,M_{\mu,\C})[(\pi'^{\infty,S})^{U^S}] \subset H^b_!(X_U,M_{\mu,\C})[(\pi'^{\infty,S})^{U^S}]$.
\end{proof}

\subsection{Selmer groups}\label{sec: selmer}
We now illustrate how to combine the results above with deformation theory to obtain some Bloch-Kato type results relating Selmer groups and $L$-functions.
We use the same definitions and notation of local and global deformation problems as in \cite[section 6.2.1]{10_authors} and we will always take $\Lambda_v=\O$ for all $v\in S$.
In particular, $\bar\rho:G_{F,S}\to \GL_n(k)$ is absolutely irreducible, $\mathcal D_v$ is a local deformation problem for each $v\in S$,
$$\mathcal S=(\bar\rho,S,\{\O\}_{v\in S},\{\mathcal D_v\}_{v\in S})$$ is a global deformation problem, and
$R_{\mathcal S}$ is the ring representing the deformation functor of type $\mathcal S$.

Fix $\rho:G_{F,S} \to \GL_n(\O)$ a lifting of $\bar\rho$ of type $\mathcal S$.
For each $m\ge 1$, let 
$$\O_m:= \O \oplus \frac{\pi^{-m}\O}{\O} \epsilon$$
with multiplication given by $(a,b\epsilon)(c,d\epsilon)=(ac,(bc+ad)\epsilon)$.
This is a local $\O$-algebra and there is a natural map $\O_m\twoheadrightarrow \O$ given by projection to the first factor.
\begin{definition}
    We let $$\mathcal L^1_v(\pi^{-m}\O/\O)$$
be the preimage of $\mathcal D_v(\O_m)$  under the isomorphism
\[Z^1\left(G_{F_v},\Ad\rho\otimes_\O \frac{\pi^{-m}\O}{\O}\right)\xrightarrow{\sim}
\{\text{liftings } G_{F_v}\to \GL_n(\O_m)\text{ of }\rho|_{G_{F_v}}\}\]
given by $c\mapsto (1+c\epsilon)\rho|_{G_{F_v}},$
where $Z^1$ means the group of continuous $1$-cocycles.
\end{definition}

We know 
\begin{equation}\label{eq:cocycle-lifting}
    Z^1(G_{F_v},\Ad\rho\otimes_\O E/\O)\xrightarrow{\sim}
\{\text{liftings } G_{F_v}\to \GL_n(\O\oplus \dfrac{E}{\O}\epsilon)\text{ of }\rho|_{G_{F_v}}\}
\end{equation}
Since\footnote{Note that $\Ad\rho\otimes_\O E/\O$ is a discrete $G_{F_v}$-module.} $Z^1(G_{F_v},\Ad\rho\otimes_\O E/\O) = \varinjlim_m Z^1(G_{F_v},\Ad\rho\otimes_\O \frac{\pi^{-m}\O}{\O})$, any lifting $G_{F_v}\to \GL_n(\O\oplus \dfrac{E}{\O}\epsilon)\text{ of }\rho|_{G_{F_v}}$ necessarily has image in $\GL_n(\O_m)$ for some $m\ge 1$.
\begin{definition}\footnote{Since $\O\oplus \frac{E}{\O} \epsilon$ is not a complete local ring, the term 'of type $\mathcal D_v$' is not defined for liftings to $\GL_n(\O\oplus \frac{E}{\O} \epsilon)$ a priori.}
    A lifting $G_{F_v}\to \GL_n(\O\oplus \dfrac{E}{\O}\epsilon)\text{ of }\rho|_{G_{F_v}}$ is of type $\mathcal D_v$ if it is of type $\mathcal D_v$ when it is regarded as a lift with codomain in $\GL_n(\O_m)$ for some $m$ (or, equivalently, for all $m$ for which $\GL_n(\O_m)$ contains the image of the lift.)
\end{definition}

The following is immediate from the definitions.
\begin{defn/lemma}
    The following subgroups of $Z^1(G_{F_v},\Ad\rho\otimes_\O E/\O)$ are equal. We denote them by $\mathcal L^1_v(E/\O)$.
    \begin{enumerate}[(i)]
        \item $\varinjlim_m \mathcal L^1_v(\pi^{-m}\O/\O)$
        \item preimage of $\{\text{liftings } G_{F_v}\to \GL_n(\O\oplus \dfrac{E}{\O}\epsilon)\text{ of }\rho|_{G_{F_v}}\text{ of type }\mathcal D_v\}$ under the isomorphism \eqref{eq:cocycle-lifting}.
    \end{enumerate}
\end{defn/lemma}

Since $\rho|_{G_{F_v}}$ is of type $\mathcal D_v$, $a\rho|_{G_{F_v}} a^{-1}$ is also of type $\mathcal D_v$ for all $a\in \ker(\GL_n(\O_m)\to \GL_n(\O))$ for all $m$ by definition of local deformation problem. It follows that $\mathcal L^1_v(\pi^{-m}\O/\O)$ and $\mathcal L^1_v(E/\O)$ both contain the group of $1$-boundaries.

\begin{definition}
    We define $\mathcal L_v(\pi^{-m}\O/\O)$ to be the image of $\mathcal L^1_v(\pi^{-m}\O/\O)$ under $Z^1\to H^1$.
    Similarly, we define $\mathcal L_v(E/\O)$ to be the image of $\mathcal L^1_v(E/\O)$ under $Z^1\to H^1$. Equivalently, by exactness of direct limits, $\mathcal L_v(E/\O)=\varinjlim_m \mathcal L_v(\pi^{-m}\O/\O).$

    We also define the Selmer groups
    \[H^1_{\mathcal S}\left(\Ad\rho\otimes_\O \frac{\pi^{-m}\O}{\O}\right):= \left\{c\in H^1\left(G_{F,S},\Ad\rho\otimes_\O \frac{\pi^{-m}\O}{\O}\right):c_v\in \mathcal L_v(\pi^{-m}\O/\O)\;\forall v\in S\right\}\]
    and
    \[H^1_{\mathcal S}\left(\Ad\rho\otimes_\O \frac{E}{\O}\right):= \left\{c\in H^1\left(G_{F,S},\Ad\rho\otimes_\O \frac{E}{\O}\right):c_v\in \mathcal L_v(E/\O)\;\forall v\in S\right\},\]
    where $c_v$ is the restriction of $c$ to $G_{F_v}$.
    It is easy to verify that 
    \[H^1_{\mathcal S}\left(\Ad\rho\otimes_\O \frac{E}{\O}\right) = \varinjlim_m H^1_{\mathcal S}\left(\Ad\rho\otimes_\O \frac{\pi^{-m}\O}{\O}\right).\]
\end{definition}

\begin{lemma}\label{lem:tangent-Selmer}
    The strict equivalence class $[\rho]$ of $\rho$ gives rise to a local $\O$-algebra homomorphism $R_{\mathcal S}\xrightarrow{\theta} \O$.
    Let $\p:=\ker\theta$. Then
    \[\Hom_\O(\p/\p^2, E/O) \cong H^1_{\mathcal S}(\Ad\rho \otimes_\O E/\O),\]
\end{lemma}

\begin{proof}
    This is well-known so we will just sketch a proof.
    Note that $\p/\p^2$ is a finitely generated $R_{\mathcal S}/\p=\O$-module, so any $\O$-algebra homomorphism $\p/\p^2 \to E/\O$ has image contained in $\varpi^{-m}\O/\O$ for some $m\ge 1$.
    Thus, if we know \[\Hom_\O(\p/\p^2, \varpi^{-m}\O/\O) \cong H^1_{\mathcal S}(\Ad\rho \otimes_\O \varpi^{-m}\O/\O)\]
    for all $m\ge 1$, then taking colimit will give the desired result.
    This follows from the chain of isomorphisms
    \begin{align}
        &\;\;\;H^1_{\mathcal S}(\Ad\rho \otimes_\O \varpi^{-m}\O/\O)\nonumber\\
        \cong &\;\{\text{liftings } G_{F,S} \to \GL_n(\O_m) \text{ of }\rho\text{ of type }\mathcal S\}/\ker(\GL_n(\O_m)\to \GL_n(\O)) \label{eq:lifts}\\
        \cong &\;\{\text{deformations } G_{F,S} \to \GL_n(\O_m) \text{ of }[\rho]\text{ of type }\mathcal S\}\label{eq: deformations}\\
       \cong &\;\{f\in\Hom_\O(R_{\mathcal S},\O_m): f\pmod\epsilon=\theta\}\nonumber\\
       \cong &\;\Hom_\O(\p/\p^2, \varpi^{-m}\O/\O).\nonumber
    \end{align}
    In \eqref{eq:lifts}, $\ker(\GL_n(\O_m)\to \GL_n(\O))$ acts on $\{\text{liftings } G_{F,S} \to \GL_n(\O_m) \text{ of }\rho\text{ of type }\mathcal S\}$ by conjugation.
    In \eqref{eq: deformations}, by deformations of $[\rho]$, we mean deformations of $\bar\rho$ whose pushforward to $\GL_n(\O)$ is $[\rho]$.
    To show the bijectivity of \eqref{eq:lifts} and \eqref{eq: deformations}, note that the centralizer of the image of $\rho$ in $\GL_n(\O)$ is $\O^\times$ by \cite[lemma 2.1.8]{CHT} and $\bar\rho$ is absolutely irreducible. 
    The other steps are easy.
\end{proof}

\begin{theorem}\label{thm: selmer}
    Let the assumptions be as at the start of \Cref{sec:4}.
    Suppose in addition that $F$ is a CM or totally real number field, and $S$ contains $\{\text{places }w: w|_{\Q} \text{ is ramified in $F$ or }w|_{\Q}=p\}$ (and all $v$ such that $\pi_v$ is ramified).

    Assume $\m$ is non-Eisenstein.
    Then there is a continuous Galois representation 
    \[\rho_\m:G_{F,S} \to \GL_n(T)\]
    such that for all $v\notin S$ of $F$, the characteristic polynomial of $\rho_\m(\Frob_v)$ is 
    \[X^n-T_{v,1}X^{n-1}+\dots+(-1)^i q_v^{i(i-1)/2}T_{v,i}X^{n-i}+\dots+ (-1)^n q_v^{n(n-1)/2}T_{v,n},\]
    where $T_{v,i}=[\GL_n(\O_{F_v})diag(\varpi_v,\dots,\varpi_v,1,\dots,1)\GL_n(\O_{F_v})]$ with $\varpi_v$ appearing $i$ times.
    
    Assume that $\rho_\m$ is a lifting of $\bar\rho_\m$ of type $\mathcal S$, where $\mathcal S=(\bar\rho_\m,S,\{\O\}_{v\in S},\{\mathcal D_v\}_{v\in S})$ is some global deformation problem. 
    Let $\rho:=\lambda\circ \rho_\m$, where $\lambda:T\to \O$ is induced from $\Lambda$ as in \Cref{lem: small lambda}.
    Then
    \begin{equation}\label{selmer-L}
        \#H^1_{\mathcal S}(\Ad\rho \otimes_\O E/\O) \ge \# (\O/L^{alg}(1,\pi,\Ad^\circ,\epsilon))
    \end{equation}
    where $\#$ denotes the order of a group.
\end{theorem}

\begin{proof}
    The key to proving the existence\footnote{If one makes extra assumptions on $F$ and $S$, then the existence of $\rho_\m$ follows immediately from \cite[theorem 2.3.7]{10_authors}.} of $\rho_\m$ is to use \cite[Theorem A]{HLTT} and Carayol's lemma.

    Since $\m$ is non-Eisenstein, $H^*_!(X_U,M_{\mu,\overline\Q_p})_\m=H^*_{cusp}(X_U,M_{\mu,\overline\Q_p})_\m$ by \Cref{lem: boundary vanishes}.
    Let $W=H^*_{cusp}(X_U,M_{\mu,\overline\Q_p})_\m$.
    By (the proof of) \Cref{lem: M ker lambda}\ref{decomp of W},
    \begin{equation}\label{decomp of W 2}
        W = \oplus_{\alpha: \T^S_{\overline\Q_p}(W) \to \overline\Q_p} W[\ker\alpha]
    \end{equation}
    where $\alpha$ runs over $\overline\Q_p$-algebra homomorphisms $\T^S_{\overline\Q_p}(W)\to\overline\Q_p$. Any such $\alpha$ satisfies $(\alpha|_{\T^S})^{-1}(\m_{\overline\Q_p})= \m$.
    By equation \eqref{decomp of H cusp} and \cite[Ch.II, Prop. 3.1]{BW_cts_cohomology}, 
    \[H^*_{cusp}(X_U,M_{\mu,\C})= \oplus_{\pi'} H^*(\mathfrak g,K_\infty^\circ , M_{\mu,\C}\otimes_\C \pi'_\infty)\otimes_\C (\pi'^\infty)^U,\]
    where $\pi'$ runs over all cohomological cuspidal automorphic representations of $\GL_n(\mathbb A_F)$ of weight $\mu$ with $\pi'^U\neq 0$.
    It follows that any $\alpha$ as above gives rise to a cohomological cuspidal automorphic representation $\pi_{\alpha}$ of $\GL_n(\mathbb A_F)$ of weight $\mu$ such that $\pi_{\alpha}^U\neq 0$ and $t\in\T^S$ acts on $\pi^U$ by $\alpha(t)$.

    We have a $\overline{\Q}_p$-algebra homomorphism
    \begin{align*}
        h:T\otimes_\O \overline\Q_p= \T^S_{\overline\Q_p}(W) &\to \prod_{\alpha:\T^S_{\overline\Q_p}(W) \to \overline\Q_p} \overline\Q_p\\
        t &\mapsto (\alpha(t))_{\alpha}. 
    \end{align*}
    By \eqref{decomp of W 2}, $h$ is injective.
    If $h$ is not surjective, then its image is a proper $\overline\Q_p$-subalgebra of $\prod_{\alpha}\overline\Q_p$, so, by a simple arguement,
    there exist $\beta,\gamma: \T^S_{\overline\Q_p}(W) \to \overline\Q_p$ such that $\beta\neq \gamma$ and the image is contained in $\{z\in\prod_\alpha \overline\Q_p: z_\beta = z_\gamma\}.$
    This is impossible by the strong multiplicity one theorem.
    Hence, $h$ is an isomorphism.

    By \cite[Theorem A]{HLTT}, for every $\alpha$, there is a unique continuous semisimple representation 
    \[ \rho_{\pi_\alpha}: \Gal(\overline F/F) \to \GL_n(\overline\Q_p)\]
    such that for all $v\not\in S$, $\rho_{\pi_\alpha}$ is unramified at $v$ and $\rho_{\pi_\alpha}(\Frob_v)$ has characteristic polynomial $\alpha(P_v(X))$, where $\alpha(P_v(X))$ is the polynomial obtained by applying $\alpha$ to every coefficient of $P_v(X):= X^n-T_{v,1}X^{n-1}+\dots+(-1)^i q_v^{i(i-1)/2}T_{v,i}X^{n-i}+\dots+ (-1)^n q_v^{n(n-1)/2}T_{v,n}$.
    By the Baire category theorem and the compactness of $\Gal(\overline F/F)$, after conjugating $\rho_{\pi_\alpha}$ by elements of $\GL_n(\overline\Q_p)$, we can assume that $\rho_{\pi_\alpha}(\Gal(\overline F/F)) \subset \GL_n(\O_L)$ for every $\alpha$, where $L$ is a finite extension of $E$. See \cite[Lemma 2.1.5]{CHT} for the proof of an analogous result.
    By the Brauer-Nesbitt theorem, every residual Galois representation $\overline\rho_{\pi_\alpha}: \Gal(\overline F/F) \to \GL_n(k_L)$, where $k_L$ is the residue field of $L$, is isomorphic to $\overline\rho_\pi$.
    We can conjugate $\rho_{\pi_\alpha}$ by elements of $\GL_n(\O_L)$ to ensure that all these residual Galois representations are equal to $\overline\rho_\pi$, not just conjugate to it.
    By enlarging $L$ if necessary, we can also assume $h(T)\subset \prod_\alpha \O_L$.

    Let
    \[\rho=(\rho_{\pi_\alpha}) : \Gal(\overline F/F) \to \GL_n(T\otimes_\O \overline\Q_p) = \prod_\alpha \GL_n(\overline\Q_p).\]
    For every $v\not\in S$, $\rho$ is unramified at $v$ and $\rho(\Frob_v)$ has characteristic polynomial $P_v(X).$

    Let $A:=\{(x_\alpha)\in\prod_\alpha \O_L: \bar x_\alpha \text{ are all equal and lies in }\O/\varpi\}$, where $\bar x_\alpha$ is the reduction mod $\m_L$ of $x$.
    This is a local ring with maxial ideal $\m_A:=\{(x_\alpha)\in\prod_\alpha \O_L: \bar x_\alpha =0 \text{ for all }\alpha\}$ and residue field $\O/\varpi$.
    This is also finitely generated as an $\O$-module, so it is a complete Noetherian ring.\footnote{It is clear that $A$ is $\varpi$-adically complete. By the going up and incomparability theorem, $\sqrt{\varpi A} = \cap_{\p\subset \varpi A}\p = \m_A$. It follows that $A$ is also $\m_A$-adically complete.}
    Thus, $A\in CNL_\O$ (the category of complete Noetherian local $\O$-algebra with residue field $\O/\varpi$).

    In the remaining part, let us denote the image of $\m$ in $T$ also as $\m$. 
    Note that $T\subset A$ and $T\in CNL_\O$.
    Also, $T\cap \m_A = \m$ by, so $T\hookrightarrow A$ is continuous.
    This is a closed map of topological spaces, because $T$ is compact\footnote{To see this, note that $T=\lim_i T/m^i$ is profinite.} and $A$ is Hausdorff.
    In particular, $T$ is a closed subset of $A$.

    Note that $\rho(\Gal(\overline F/F)) \subset \GL_n(A)$, so we can view $\rho$ as a representation
    \[\rho : \Gal(\overline F/F) \to \GL_n(A).\]
    The residual Galois representation $\bar\rho:\Gal(\overline F/F) \to \GL_n(A/\m_A)$ is absolutely irreducible because $\m$ is non-Eisenstein.
    We know that $\mathrm{tr}(\rho(\Frob_v))= \pm T_{v,1} \in T$ for all $v\not\in S$.
    By Chebotarev density theorem and the fact that $T\subset A$ is a closed subset, $\mathrm{tr}\rho(\Gal(\overline F/F)) \subset T$.
    By Carayol's lemma (\cite[Lemma 2.1.10]{CHT}), there exists $g\in \GL_n(A) \to \GL_n(\O/\varpi)$ such that $g \rho(\Gal(\overline F/F)) g^{-1} \subset \GL_n(T)$.
    Then $g\rho g^{-1}$ is unramified outside $S$, so it gives rise to
    \[\rho_\m:G_{F,S} \to \GL_n(T)\]
    such that for all $v\notin S$ of $F$, the characteristic polynomial of $\rho_\m(\Frob_v)$ is $P_v(X)$. This proves the first part.

    For the second part, assume that $\rho_\m$ is of type $\mathcal S$. Its strict equivalence class induces an $\O$-algebra homomorphism $f:R_{\mathcal S}\to T$.
    We know that $\T^S$ is generated by 
    $$\{T^i_v,(T^n_v)^{-1}: v\notin S, 1\le i\le n\}$$ as an $\O$-algebra\footnote{This can be easily deduced by applying the Satake isomorphism to $\H(\GL_n(F_v),\GL_n(\O_{F_v}))\otimes_\Z \Z[q_v^{1/2}]$.}.
    For all $g\in G_{F,S}$, every coefficient of the characteristic polynomial of $\rho_\m(g)$ is in the image of $R_{\mathcal S}\to T$.
    Taking $\Frob_v^{\pm 1}$, we know $f$ is surjective.

    If $H^1_{\mathcal S}(\Ad\rho \otimes_\O E/\O)$ is infinite, then equation \eqref{selmer-L} is trivial.
    Suppose $H^1_{\mathcal S}(\Ad\rho \otimes_\O E/\O)$ is finite.
    Let $\theta=\lambda\circ f$ and $\p:=\ker\theta$.
    It follows from \Cref{lem:tangent-Selmer} that $\p/\p^2$ is also finite and 
    \[\#H^1_{\mathcal S}(\Ad\rho \otimes_\O E/\O) = \#(\p/\p^2).\]
    By \cite[page 141 equation (5.2.3)]{DDT},
    \[\#(\p/\p^2) \ge \#(\O/\eta_{R_{\mathcal S}})\]
    where $\eta_{R_{\mathcal S}}:=\theta(Ann_{R_{\mathcal S}}(\p))$.
    By \cite[page 140 equation (5.2.2)]{DDT}, 
    $$\#(\O/\eta_{R_{\mathcal S}}) \ge \#(\O/\eta_T),$$ where $\eta_T=\lambda(Ann_T(\ker\lambda))$, which is the same as $\eta_\pi$ in the previous subsections.
    By \Cref{lem: M lambda rk 1}, \Cref{lem: boundary vanishes}, and \Cref{thm: main}, 
    $$\#(\O/\eta_\pi) \ge \#(\O/\eta_{\pi,b,\epsilon}) \ge \# (\O/L^{alg}(1,\pi,\Ad^\circ,\epsilon)).$$
    We get the desired inequality.
\end{proof}

As an illustration of \Cref{thm: selmer}, let us give an example.
Let $\mathcal D_v^\Box$ be the functor on $CNL_\O$ (category of complete Noetherian local $\O$-algebras with residue fields $k$) that sends $A$ to the set of all lifts of $\bar\rho_\m|_{G_{F_v}}$ to $A$.
In the Fontaine-Laffaille case, if $v$ is a $p$-adic place, we let $\mathcal D_v^{FL}$ be the local deformation problem that sends any $A\in CNL_\O$ that is finite over $\O$ to all liftings of $\bar\rho_\m|_{G_{F_v}}$ to $A$ that are Fontaine-Laffaille of type $(\mu_\tau)_{\tau\in Hom(F_v,E)}$. See \cite[sections 4.1, 6.2.14]{10_authors} for more detail.
\begin{corollary}\label{coro: FL}
    Let the assumptions be as at the start of \Cref{sec:4}.
    Suppose in addition that 
    \begin{itemize}
        \item $F$ is a CM or totally real number field.
        \item $S$ contains $\{\text{places }w: w|_{\Q} \text{ is ramified in $F$ or }w|_{\Q}=p\}$ (and all $v$ such that $\pi_v$ is ramified).
        \item $U_v=\GL_n(\O_{F_v})$ for all $v\mid p$.
        \item the residual Galois representation $\bar\rho_\pi: \Gal(\overline F/F) \to \GL_n(\overline F_p)$ is irreducible and decomposed generic \cite[Definition 4.3.1]{10_authors}.
    \end{itemize}
    Then for all $v\mid p$, $\rho_m|_{G_{F_v}}$ is Fontaine-Laffaille of type $(\mu_\tau)_{\tau\in Hom(F_v,E)}$, where $\rho_\m$ is the Galois representation constructed in \Cref{thm: selmer}. In particular, $\rho_m$ is a lifting of $\bar\rho_\m$ of type $\mathcal S=(\bar\rho_\m,S,\{\O\}_{v\in S}, \{\mathcal D_v^{FL}\}_{v\mid p}\cup\{\mathcal D_v^\Box\}_{v\in S-\{v\mid p\}})$ and $$\#H^1_{\mathcal S}(\Ad\rho \otimes_\O E/\O) \ge \# (\O/L^{alg}(1,\pi,\Ad^\circ)).$$
\end{corollary}

\begin{proof}
    This follows from the proof of \Cref{thm: selmer} and by replacing the use of \cite[Theorem A]{HLTT} by \cite[Theorem 4.3.1]{Caraiani-Newton}.
\end{proof}

\begin{remark}
    To relate this Selmer group to the Bloch-Kato Selmer group, see \cite[lemma 2.1]{DFG} and \cite[section 7.3]{Dimitrov}. In their setup, the $\O$-Fitting ideal of their Selmer group was equal to that of the Bloch-Kato one multiplied by $\prod_{v\in \Sigma}Fitt_\O (H^1_f(F_v,(\Ad^\circ\rho_f)^*(1)))$ for some finite set of places $\Sigma$. Each term in the product was then shown to be equal to the local Tamagawa number divided by the local $L$-factor at that place. Combining this with a suitable $R=T$ theorem, they were able to deduce a form of the Bloch-Kato conjecture using similar argument to the proof of \Cref{thm: selmer}.
\end{remark}

\section*{Acknowledgement}
I would like to thank my supervisor Tobias Berger for his help with this project, James Newton and Robert Kurinczuk for lots of constructive feedback which greatly enhances the quality of this paper, Anantharam Raghuram for sketching the proof of \Cref{lem:T acts semisimply} part \ref{isotypic} in an email, Haluk Sengun for suggesting nice references for learning about cohomology of arithmetic groups, and an anonymous referee for giving many useful comments on the previous version of this paper.
This work was completed while in receipt of the EPSRC Doctoral Training Partnership (DTP) Studentship with grant number EP/W524360/1.

\bibliographystyle{amsalpha}
\bibliography{reference}
\end{document}